\documentclass[11pt]{article}

\usepackage[T1]{fontenc}
\usepackage[utf8]{inputenc}

\usepackage[a4paper, margin=2cm]{geometry}
\usepackage{amsmath}
\usepackage{amssymb}
\usepackage{amsthm}
\usepackage{bbm}
\usepackage[title]{appendix}
\usepackage{xcolor}

\newcommand{\dd}{\mathrm{d}}
\newcommand{\ddiv}{\mathrm{div\,}}
\renewcommand{\div}{\mathrm{div}}
\newcommand{\rot}{\mathrm{rot}}
\newcommand{\cP}{\mathcal{P}}
\newcommand{\cQ}{\mathcal{Q}}

\def\XXint#1#2#3{{\setbox0=\hbox{$#1{#2#3}{\int}$}
    \vcenter{\hbox{$#2#3$}}\kern-.5\wd0}}

\newtheorem{tw}{Theorem}
\newtheorem{lem}[tw]{Lemma}
\newtheorem{prop}[tw]{Proposition}
\newtheorem{rem}[tw]{Remark}

\numberwithin{equation}{section}
\numberwithin{tw}{section}

\title{A multifluid model with chemically reacting components -- construction of weak solutions}
\author{Piotr B. Mucha$^1$, \v{S}\'{a}rka Ne\v{c}asov\'{a}$^2$, Maja Szlenk$^1$}
\date{}

\begin{document}

\maketitle

\begin{center}
1. University of Warsaw, Institute of Applied Mathematics and Mechanics \\
ul Banacha 2, 02-097 Warszawa, Poland \\
Email: p.mucha@mimuw.edu.pl, m.szlenk@uw.edu.pl\\
2.Institute of Mathematics, Czech Academy of Sciences, \\
\v Zitn\'a 25, 115 67 Praha 1, Czech Republic\\
Email: matus@math.cas.cz

\end{center}

\begin{abstract}
   We investigate the existence of weak solutions to a multi-component system, consisting of compressible chemically reacting components, coupled with the  compressible Stokes equation for the velocity. Specifically, we consider the case of irreversible chemical reactions and assume a nonlinear relation between the pressure and the particular densities. These assumptions cause the additional difficulties in the mathematical analysis, due to the possible presence of vacuum.  
   
   It is shown that there exists a global weak solution, satisfying the $L^\infty$ bounds for all the components. We obtain strong compactness of the sequence of densities in $L^p$ spaces, under the assumption that all components are strictly positive. The applied method captures the properties of models of high generality, which admit an arbitrary number of components. Furthermore, the framework that we develop can handle models that contain both diffusing and non-diffusing elements.
\end{abstract}

\noindent
\textbf{keywords:} compressible Stokes system, multi-component flow, weak solutions, irreversible chemical reaction


\section{Description of the multifluid system}
In recent years, a significant progress has been done in the mathematical theory of mixtures, both compressible and incompressible. The physical background for the modeling of mixtures can be found in \cite{Pr}, an approach closer to modern understanding of continuum mechanics and thermodynamics  in \cite{RaTa}. Let us mention also the book \cite{Gi}, where a general model of compressible chemically reacting mixtures under very general conditions is analyzed from the mathematical point of view, however, only in a small neighborhood of a given static solution. We will be more concerned about results which are not restricted to small data. As the results differ with respect to studied complexity, we mention here for the case of incompressible fluids the paper \cite{Ro} and the book \cite{Ot}.

% and the paper \cite{FePeTr} for the case of compressible chemically reacting mixtures which is one of the first results in this field.
% (however, under quite restrictive assumptions on the structure of the pressure and entropy)
In the paper we analyze the system of $N$ compressible, reacting fluids with densities $$\varrho_1, \dots, \varrho_N\colon [0,T]\times\Omega \; \subset \; \mathbb{R}^3 \to \mathbb{R},$$ moving along a common velocity vector $u\colon [0,T]\times\Omega\to \mathbb{R}^3$, where $\Omega\subset\mathbb{R}^3$ is a bounded, open set with $C^{2,\theta}$ boundary, $0<\theta\leq 1$.  We assume that the motion is slow,
%low Reynolds number regime, 
therefore the momentum equation is given by the Stokes equation.
%\[ -\mu\Delta u -\nabla((\mu+\lambda)\ddiv u) + \nabla %p(\vec\varrho\,) = 0, \]
%where the first viscosity parameter $\mu$ is constant and the second viscosity $\lambda$ depends on $\ddiv u$. 
The pressure $p$ depends on the vector of the densities $\vec\varrho=(\varrho_1,\dots, \varrho_N)$.
The interactions between the particular components are described by the mass production rates $\omega_1,\dots,\omega_N\colon\mathbb{R}^N_+\to\mathbb{R}$, $\omega_1,\dots,\omega_N\in C^1(\overline{\mathbb{R}_+^N})$ and the mass diffusion is described by fluxes $F_i$. 

The full system, incorporating all these phenomena reads
\begin{equation}\label{diff_react}
    \begin{aligned}
    \partial_t\varrho_{i} + \ddiv(\varrho_{i} u) -\ddiv F_i &= \omega_{i}(\vec\varrho\,), \quad i=1,\dots,N , & \mbox{in \ }\Omega \times (0,T),\\
   -\mu\Delta u -\nabla((\mu+\lambda)\ddiv u) + \nabla p(\vec\varrho\,) &= 0, & \mbox{in \ } \Omega \times (0,T),
    \end{aligned}
\end{equation}
where we assume that $\mu$, $\lambda$ are constants such that $\mu>0$ and $\lambda + \frac{2}{3}\mu >0$.
\begin{rem}
    Note that in the system (\ref{diff_react}), one should also consider the force terms $\varrho f+ g$ on the right hand side of the momentum equation. To avoid unnecessary complications, we set $f=g=0$, however our result can be modified to include nonzero vector fields $f$ and $g$ with sufficiently high integrability.
\end{rem}

%\subsection{Initial and boundary conditions}

The system is considered with the slip boundary conditions for $u$ and the Neumann type constrain for the fluxes, namely on $\partial\Omega \times (0,T)$ 
\begin{equation}\label{slip}
    \begin{aligned}
    u\cdot n &= 0  & \mbox{at \ } \partial \Omega \times (0,T),\\
    n\cdot \mathbb{T}(u,p)\cdot\tau_k + fu\cdot\tau_k &= 0, \quad k=1,\dots,2, & \mbox{at \ } \partial \Omega
    \times (0,T),\\
    \quad F_i\cdot n & = 0, \quad i=1,...,N,& \mbox{at \ } \partial \Omega \times (0,T),
    \end{aligned}
\end{equation}
where $n$ and $\tau_k$ are the normal and tangent vectors respectively, $f> 0$ is constant and describes the friction, and the stress tensor $\mathbb{T}$ is given by
\[ \mathbb{T}(u,p(\vec \varrho\,))=2\mu\mathbb{D}u+(\lambda\ddiv u-p(\vec \varrho\,))\mathbb{I}. \]
The system is finally supplemented by the set of initial data for the species' densities
%For the components $\varrho_i$, we have the initial conditions:
\[ {\varrho_{i}}_{|_{t=0}} = \varrho_{0,i}. 
%\in L^\infty(\Omega). 
\]
% Moreover, the initial densities have to satisfy
% \[ \varrho_0=\sum_{i=1}^N \varrho_{0,i} \in W^{1,2}(\Omega). \]
%Moreover, for the boundary conditions on $\varrho_i$ we assume that
%\[ F_i\cdot n = 0. \]

%\subsection{The pressure and $F_i$}
The core of the system is the constitutive relation giving us the form of the pressure. We assume that the total pressure $p$ is in the form
\begin{equation}\label{pressure} p(\vec\varrho\,)=\sum_{i=1}^N p_i(\varrho_i), \mbox{ \ with \ }
p_i(\varrho_i\,)=\frac{1}{m_i}\varrho_i^{\gamma_i}, \end{equation} 
where $m_i$'s correspond to molar masses of the components and $\gamma_i>1$. 
The diffusion fluxes are depending on the pressure. We assume that the fluxes have the form 
\begin{equation}\label{flux}
 F_i = \nabla p_i -\frac{\varrho_i}{\varrho}\nabla p. 
 \end{equation}
This formula was introduced in Chapter 2.5 in \cite{Gi}, in the case of linear pressure. It  guarantees the fundamental feature of the fluxes, namely that 
$\sum_{i=1}^N F_i=0$. It also provides non-negativity of densities $\varrho_i$, which shall be explained later in the paper.
%
% {\bf Piotr: I have no idea if is it a good place for it}
% In consequence
% \begin{equation}\label{F_i}
% \begin{aligned} \sum_{i=1}^N  \frac{1}{m_i}\nabla\varrho_i\cdot F_i &=
% \frac{1}{2}\sum_{i=1}^N \frac{1}{\varrho_i}\left(\frac{1}{m_i}\nabla\varrho_i^2-\frac{\varrho_i}{\varrho}\nabla p\right)\cdot F_i \\
% &= \frac{1}{2}\sum_{i=1}^N \frac{1}{\varrho_i}|F_i|^2.
% \end{aligned} \end{equation}
%
%\subsection{Assumptions on $\omega_i$}
%
Regarding the reaction terms,
we  focus on {\it the irreversible reaction} of the form
 \begin{equation}\label{irreversible} A_1 + ...+ A_K \rightarrow C_1+ ... + C_L. 
 \end{equation}
We impose the conditions on $\omega_i$ so that the densities remain non-negative for non-negative initial densities, and that the total mass of the system is preserved. To obtain the first property, we have to assume
\begin{equation}\label{omega1} \omega_i(\vec\varrho\,)\geq 0 \quad \text{for} \quad \varrho_i=0 \quad \text{and} \quad \varrho_j\geq 0, \; j\neq i. \end{equation}
Moreover, the conservation of the total mass is satisfied if 
\begin{equation}\label{omega2}
\sum_{i=1}^N\omega_i(\vec\varrho\,)=0. \end{equation}
Condition (\ref{omega2}) is the basic rule which saves us the conservation of the total mass.
This, together with the fact that $\sum_{i=1}^N F_i=0$, ensures that the sum of the densities $\varrho=\sum_{i=1}^N\varrho_i$ satisfies the continuity equation
\[ \partial_t\varrho+\ddiv(\varrho u) = 0 \]
and as a consequence the total mass $\int_\Omega \varrho \; \dd x$ remains constant.

\subsection{Main results}

The problem of existence of weak solutions to system involving degenerate diffusion remains up to now not fully resolved.
In the literature we find some positive answers to similar questions, but for slightly or significantly different systems. In \cite{Za_1,zatorska2013}
the author was able to control the fluxes $\bar F_i$ in terms of gradients of species, but in that framework 
the densities have $L^2$-integrable gradients. It is a consequence of a very special dependence of viscous coefficients on the density. Here, there is no such possibility. In \cite{degenerate-parabolic}, the authors construct the weak solutions to a system of continuity equations involving chemical reactions with a given velocity. In that case, they overcome the lack of regularity for the components by assuming that the total density is in $L^2(0,T;H^1)$, which together with the linear pressure provides the estimates on the gradients of the particular species. However, we are not able to improve the regularity of the total density naturally from the equation.

% To underline the nontriviality of Theorem \ref{th:1.1}, we need to refer to the auxillary result of the paper, which
% obtains the solutions to an approximate system to (\ref{diff_react}). Here the exponents $\gamma_i$ can be different from $2$:

% \textcolor{orange}{Does the term "strictly elliptic" below make sense in this context? [M]}

Because of that, we need to include additional assumptions on our solution. First, let us present the intermediate result, concerning the situation where the fluxes $F_i$ contain a small  elliptic regularization. We, then, are able to show:

\begin{tw}\label{appr_th}
Let $\varepsilon>0$.
    Assume that $p_i(\varrho_i)=\frac{1}{m_i}\varrho_i^{\gamma_i}$, where the exponents $\gamma_i$ satisfy the relation
    \begin{equation}\label{gammas} 2\gamma_{\max} < 3\gamma_{\min} - \gamma_S+1 \end{equation}
    for $\gamma_{\max}$ and $\gamma_{\min}$ being the maximum and minimum of $\{\gamma_1,\dots,\gamma_N\}$ respectively, and $\gamma_S=\max_{j\in S}\gamma_j$ with $S$ denoting the components on the right hand side of (\ref{irreversible}). Let $F_i$ be given by (\ref{flux}). Then, for any $T>0$, there exists at least one weak solution $(\varrho_1,\ldots,\varrho_N,u)$ to the system
    \begin{equation}\label{approx}
    \begin{aligned}
    \partial_t\varrho_{i} + \ddiv(\varrho_{i} u) -\ddiv F_{i} &= \omega_{i}(\vec\varrho\,) + \varepsilon\Delta\varrho_i, \quad i=1,\dots,N ,\\
   -\mu\Delta u -\nabla((\mu+\lambda)\ddiv u) + \nabla p(\vec\varrho\,) &= 0,
    \end{aligned}
\end{equation}
in $[0,T]\times\Omega$ with 
\[ {\varrho_i}_{|_{t=0}} = \varrho_{0,i,\varepsilon}\in C^\infty(\Omega), \quad \varrho_{0,i,\varepsilon}\to \varrho_{0,i} \quad \text{in} \quad L^2(\Omega), \]
the Neumann boundary condition $\frac{\partial\varrho_i}{\partial n}=0$ on $\partial\Omega$ 
and the slip boundary conditions (\ref{slip}) for $u$. Moreover, the obtained solution satisfies
\[ \varrho_i \geq 0, \quad i=1,\ldots,N, \]
\[ \|\varrho_i\|_{L^\infty((0,T)\times\Omega)}, \left\|\frac{1}{\sqrt{\varrho_i}}F_i\right\|_{L^2((0,T)\times\Omega)}, \sqrt{\varepsilon}\|\varrho_i\|_{L^2(0,T;H^1)} \leq C, \quad i=1,\ldots,N \]
and 
\[ \|u\|_{L^2(0,T;H^1)}, \|\nabla u\|_{L^\infty(0,T;BMO)}, \|\ddiv u\|_{L^\infty([0,T]\times\Omega)} \leq C, \]
where the constant $C$ does not depend on $\varepsilon$ or $T$.
\end{tw}

Note that the above result delivers well defined objects in system (\ref{approx}). The gradients of the species are in $L^2$ and thus $F_i$'s given by (\ref{flux}) make sense. The proof of Theorem \ref{appr_th} is based on the classical techniques. First, we consider a suitable regularization, which makes our fluxes less degenerate, and densities better integrable. Then starting from the Galerkin methods we prove existence of each step of approximation scheme, obtaining the solutions to (\ref{approx}) finally. The statement of Theorem \ref{appr_th} belongs to the standard theory of weak solutions to equations for compressible fluids. 
% However, the step towards Theorem \ref{th:1.1} seems to be not. We need to take $\varepsilon \to 0$, and then determine the limits in the equation. What is unexpected, is the strong convergence of all densities. We prove
% \begin{equation}\label{bb4}
%     \varrho^{(\varepsilon)}_i \to \varrho_i \mbox{ \ \ strongly in \ \ }
%     L^p(\Omega \times (0,\infty)).
% \end{equation}
% In order to justify the above limits we need to apply techniques of convergence adopted from \cite{BJ2}. Thanks to the structure of fluxes, limited in this part just to $\gamma_i=2$, we are able to prove
% (\ref{bb4}). Note that other tools from compensated compactness theory do not fit to our problem. This points out the nontriviality of Theorem \ref{th:1.1}. It is also an interesting example of the application of the Bresch \& Jabin technique, where all classical compactness methods seem to fail. Another case of the multicomponent setting where this method becomes useful is the result of D\k{e}biec et al. \cite{debiec-perthame}, concerning two-species system modelling tumor growth.

One of the main difficulties appearing with passing to the limit with $\varepsilon$ in (\ref{approx}) is possible lack of strict positivity of the particular densities. It turns out that even though we still cannot control the gradients of the densities, assuming that all components are bounded away from zero we can rewrite the fluxes $F_i$ in terms of new variables, which are already bounded in $L^2(0,T;H^1)$. We summarize this result below:
\begin{tw}\label{positive_th}
    Let $p_i$ and $\gamma_i$ be as in Theorem \ref{appr_th}. Let $(\vec\varrho_\varepsilon,u_\varepsilon)$ be a sequence of solutions to (\ref{approx}), such that for every $i=1,\dots,N$
    \[ \varrho_{i,\varepsilon} >c \mbox{ \ \ in \ \ }  [0,T]\times \Omega , \] 
    for some $c>0$. Then $(\vec\varrho_\varepsilon,u_\varepsilon)$ converges to a solution to (\ref{diff_react}), with 
    \[ F_i = \nabla p_i - \frac{\varrho_i}{\varrho}\nabla p. \]
\end{tw}

The proof is based on the fact that even though the fluxes $F_i$ form a degenerate system, we can rewrite them as
\begin{equation}\label{fluxes2}
F_i = \sum_{j=1}^{N-1}b_{i,j}\nabla q_j, \end{equation}
where $\vec q=(q_1,\dots,q_N)$ is a certain projection of $\left(\frac{\gamma_1}{\gamma_1-1}\frac{p_1}{\varrho_1},\dots,\frac{\gamma_N}{\gamma_N-1}\frac{p_N}{\varrho_N}\right)$ on the $N-1$-dimensional space orthogonal to $(1,\dots,1)$ (in particular $q_N=-\sum_{j=1}^{N-1}q_j$). It is shown that in the case when $\varrho_1,\dots,\varrho_N\neq 0$, the coefficients $b_{i,j}$ from the relation (\ref{fluxes2}) form an invertible matrix, which in turn provides an $L^2$ bound on $\nabla q_i$.
Similar approach, with a different definition of the fluxes, was considered for example in \cite{bulicek-pokorny,DDGG}. \\

It is worth noting that our method also allows the system to contain both diffusive and non-diffusive components. Assuming $N_1$ is the number of diffusing elements, the full system reads
\begin{equation}\label{diff_nondiff}
    \begin{aligned}
    \partial_t\varrho_{i} + \ddiv(\varrho_{i} u) -\ddiv F_i &= \omega_{i}(\vec\varrho\,), \quad i=1,\dots,N_1, \\
    \partial_t\varrho_j + \ddiv(\varrho_j u) &= \omega_j(\vec\varrho\,), \quad j=N_1+1,\dots,N, \\
   -\mu\Delta u -\nabla((\mu+\lambda)\ddiv u) + \nabla p(\vec\varrho\,) &= 0.
    \end{aligned}
\end{equation}

Let us now briefly explain the outline of the rest of the paper and discuss the emerging obstacles:
\begin{itemize}
    \item In Sections 2--\ref{limit_sect}, we consider the case when diffusion occurs for all the components and the most general diffusive -- non-diffusive case is shortly commented in Section 5.
    \item In Section 2, we present the a priori estimates for the solutions. In particular, we show that the densities $\varrho_i$ are in $L^\infty((0,T)\times\Omega)$. The used strategy allows us to derive the required estimates in the case where the exponents $\gamma_i$ satisfy the relation (\ref{gammas}).
    \item Section 3 is devoted to the construction of solutions to equation (\ref{approx}) and in consequence ends the proof of Theorem \ref{appr_th}.
    The next step towards the proof of Theorem \ref{positive_th} is the limit passage with $\varepsilon\to 0$. This final step is done Section \ref{limit_sect}.

    % \item Section 4 contains the end of the proof of Theorem \ref{th:1.1}. 
    % % Due to the presence of the additional terms in the continuity equations, the standard Lions--Feireisl approach \cite{lions, feireisl} is not well-suited for this framework.
    % Instead of the standard Lions--Feireisl approach \cite{lions, feireisl}, we adapt the method of Bresch and Jabin \cite{bresch-jabin,BJ2}, based on the Kolmogorov compatness criterion, which allows us to show the strong convergence of the densities. However, due to the limitations of this method we need to restrict to the case $\gamma_1=\dots,\gamma_N=2$.
    % \item In Section \ref{limit_sect} we perform the final limit passage in the framework of Theorem \ref{positive_th}, namely assuming that the sequence of approximate solutions is strictly positive. For the exponents $\gamma_i$, we assume they satisfy (\ref{gammas}).
    
\end{itemize}

\subsection{Bibliographical remarks}
The modelling of multicomponent flows dates back to XIXth century and the works of A. Fick, who studied the connection between the mass flux and molar concentrations. The relation called nowadays the Fick's law states
\begin{equation}\label{Fick}
J_i^{\rm mol}= - D_i \nabla c_i,
\end{equation}
where $J_i^{\rm mol}$ is the molar mass flux (corresponding to $F_i$ in our setting), $c_i$ denotes the molar concentration of a constituent $A_i$, and $D_i$ are diffusion coefficients, $i\in\{1,\dots,N\}$.
% Let us begin with the seminal work by A. Fick, who summarized on diffusion matter in liquids and formulated in one dimension relation between the molar mass flux $J_i^{\rm mol}$, the molar concentration of constituent $A_i$ denoted by $c_i$, $i\in \{1,...,N\}$, the phenomenological coefficient $D_i$ (the diffusion coefficient). Nowadays, this relation is called Fick's first law and has a form
% \begin{equation}\label{Fick}
% J_i^{\rm mol}= - D_i \nabla c_i.
% \end{equation}

It was realized later on that such relation is less realistic, which led to introducing of the general setting of the Theory of irreversible processes. We  refer to De Groot and Mazur \cite{DG-M}, where the deficiencies of Fick's law are described together with a new law - the Fick-Onsager form of diffusion fluxes. The diffusion coefficients $D_i$ are replaced by a matrix phenomenological coefficients $L_{ij}$ called Osanger coefficients for diffusion. Another approach is the Maxwell-Stefan equations - in the so-called diffusive approximation, which ignore the acceleration of the relative motion, see e. g. Dreyer, Druet, Gajewski and  Guhlke \cite{DDGG}. For more details concerning modelling, we refer to the article by Bothe and Druet \cite{BD}.

Maxwell-Stefan type models  were studied from the engineering point of view but there are only a few results concerning the mathematical theory. Recently such problems attract a lot of attention, see works of \cite{B,BGS,JS}. The mass-based Maxwell-Stefan approach to one-phase multi-component reactive mixture is analyzed \cite{HMPW}. It is shown the local well-posedness in $L^p$ setting and generate a local semiflow on its natural state space. A self-contained introduction to this approach can be find in \cite{J}.
Analysis of cross-diffusion systems for fluid mixtures driven by a pressure gradient  was analysis for weak and strong solution by Druet and J\"ungel \cite{DJ}.

Many mathematical results have been proved in direction of multi-component diffusion systems, where the components share the same velocity and the Maxwell-Stefan system is coupled with the Navier-Stokes equations.
The model for the compressible chemically reacting mixtures,  developed by Giovangigli in \cite{Gi}, was considered in the context of the Fick's law in \cite{FPT}; for non-diagonal mobility matrix see also \cite{Za_1}. 
    The paper \cite{MPZ_SIMA} includes the case when the mean viscosity depends on the total density and fulfils the Bresch -- Desjardins relation.   The steady problems (with non-diagonal mobility matrix) were studied in \cite{Zatorska_2011,GPZ,PiPo1}. More complex situation is considered in \cite{DDGG}. Strong solution of such type of model was investigated by Piasecki, Shibata, Zatorska \cite{PSZ1,PSZ2}.
Let us also mention  the recent work of Druet concerning the singular limit for multicomponent models, \cite{D}.

     %Other source of interesting problems are the recent results dealing with bi-fluid (and more generally multifluid) models considered e.g. in \cite{BrHand} and \cite{NP_ARMA,KKNN}.

The above results in principle concern the situation, where the chemical reaction is reversible. Concerning the irreversible reaction, from the mathematical point of view not much is known. A fast irreversible reaction of type (\ref{abc}) was investigated by Bothe and Pierre in the case of reaction-diffusion system \cite{BP}, where it appears between species $A$ and $B$ with similar concentrations. Purely numerical and experimental results, involving irreversible reactions, were obtained for example in \cite{bothe-kroeger-warnecke,bubbles1,bubbles2}. 

\section{A priori estimates}\label{a_priori_sect}

Now we proceed with a priori estimates. 
We assume for simplicity that $N=3$ and that we deal with a reaction of the type
 \begin{equation}\label{abc} A + B \rightarrow C, \end{equation}
 however from the proof it is clear that the reasoning can be adjusted to a single irreversible reaction with arbitrary number of components. For the one-way reaction, the production rates have fixed signs. The ones corresponding to reagents are negative, whereas the ones corresponding to products are positive. In the special case (\ref{abc}), $\omega_i = -\alpha_i\varrho_1\varrho_2$ for $i=1,2$ and $\omega_3 = (\alpha_1+\alpha_2)\varrho_1\varrho_2$, where $\alpha_i>0$ depend on the reaction rate and molecular weights of $A$ and $B$ respectively. The above example can be easily generalized to the case
 \[ A_1 + ...+ A_K \rightarrow C_1+ ... + C_L, \]
 with suitable definitions of $\omega_i$'s.

Throughout this section we will assume that $(\vec\varrho,u)$ is a sufficiently smooth solution to (\ref{diff_react}). We show that
 \begin{lem}\label{oszac1}
Assuming $\varrho_{0,i}\geq 0$ for $i=1,\dots,N$ and that $\Omega$ is sufficiently smooth (e. g. $\partial\Omega\in C^{2,\theta}$), then $(\vec\varrho,u)$ satisfies
 \[ \rho_i \geq 0, \qquad \int_{\Omega} \rho (t)\;\dd x = \int_\Omega \varrho_0 \;\dd x,\]
 \[ \sum_{i=1}^N\|\varrho_i\|_{L^\infty([0,T]\times\Omega)} + \sum_{i=1}^N\left\|\frac{1}{\sqrt{\varrho_i}}F_i\right\|_{L^2([0,T]\times\Omega)} + \|\ddiv u\|_{L^\infty([0,T]\times\Omega)} \leq C, \]
 where the constant $C$ depends only on $\|\vec\varrho_0\|_{L^\infty(\Omega)}$.
 \end{lem}

%\subsection{Conservation of mass and nonnegativity of the components}\label{a_priori_nonnegative}

{\bf Proof.} The proof is split into four steps.
\smallskip

{\bf Step I.} 
First, let us prove the basic properties of the solution:
mass conservation and non-negativity of all densities. As the total density satisfies the continuity equation, after integrating over $\Omega$ we get
$\frac{\dd}{\dd t}\int_\Omega \varrho \;\dd x = 0$,
and thus $\displaystyle\int_\Omega \varrho(t,x)\dd x = \int_\Omega \varrho_0\;\dd x$ for all $t>0$.

We  now show that $ \varrho_i\geq 0$  for all $ i=1,\dots,N.$
% \begin{lem}
%     If $\varrho_1,\dots,\varrho_N, u$ are smooth enough, i.e. $\varrho_i\in L^2(0,T;W^{1,2})\cap L^\infty((0,T)\times\Omega)$, $\partial_t\varrho_i \in L^2(0,T;H^{-1})$, $u\in L^2(0,T;W^{1,2})$, $\ddiv u\in L^\infty((0,T)\times\Omega)$, and the initial conditions satisfy 
%     \[ \varrho_{i,0}\geq \quad \forall_{i=1,\dots, N}, \quad \varrho_0\geq c>0, \]
%     then
%     \[ \varrho_i \geq 0 \quad \text{a.e. in} \quad (0,T)\times\Omega, \;\; i=1,\dots,N. \]
% \end{lem}
For a fixed $i$, let  $\Omega_i^-(t)=\{x\in \Omega: \varrho_i(t,x)<0\}. $
Integrating the equation on $\varrho_i$ over $\Omega_i^-$, we get
\[ \int_{\Omega_i^-(t)}\partial_t\varrho_i\dd x = \int_{\partial\Omega_i^-(t)}\left(-\varrho_iu + \frac{\gamma_i}{m_i}\varrho_i^{\gamma_i-1}\nabla\varrho_i-\frac{\varrho_i}{\varrho}\nabla p\right)\cdot \boldsymbol{\mathfrak{n}}_i^-(t)\dd S + \int_{\Omega_i^-(t)}\omega_i \dd x, \]
where $\boldsymbol{\mathfrak{n}}_i^-(t)$ is the outer normal vector to $\Omega_i^-(t)$. To deal with the last term let us observe:

\begin{rem}\label{omega_rem}
    The assumption (\ref{omega1}) allows us to extend $\omega_i$ to a function $\tilde\omega_i$ on $\mathbb{R}^N$ in such a way that $\tilde\omega_i(\vec\varrho\,)\geq 0$ whenever $\varrho_i<0$ and $\tilde\omega_i$ is locally Lipschitz on $\mathbb{R}^N$. For example, define $\tilde\omega_i(\vec\varrho\,)=\omega_i(s(\vec\varrho\,))$ on the set $\{\vec\varrho\in\mathbb{R}^N: \varrho_i\geq 0\}$ and $\tilde\omega_i(\vec\varrho\,)=\max(0,\omega_i(s(\vec\varrho\,)))$ if $\varrho_i<0$, where $s(z_1,\dots,z_N)=(|z_1|,\dots, |z_N|)$. Since $\omega_i\geq 0$ if $\varrho_i=0$, the Lipschitz continuity of $\tilde\omega_i$ is preserved.
% If $\omega_i=0$ for $\varrho_i=0$, we simply put $\tilde\omega_i=0$ for $\varrho_i<0$. If $\omega_i>0$ for $\varrho_i=0$ and $\varrho_j\geq 0$, then we first extend continuously the function 
% ${\omega_i}_{|_{\varrho_i=0}}$ so that
% ${\omega_i}_{|_{\varrho_i=0}}\geq 0$ for all $\varrho_j\in\mathbb{R}^N$, $j\neq i$. Then, for $\varrho_i<0$ we just put $\tilde\omega_i(\varrho_1,\dots,\varrho_i,\dots,\varrho_N)=\omega_i(\varrho_1,\dots,0,\dots\varrho_N)$. For simplicity we will just denote $\tilde\omega_i$ by $\omega_i$.
\end{rem}

Therefore, as by Remark \ref{omega_rem} $\omega_i\geq 0$ for $\varrho_i\leq 0$ and  ${\varrho_i}_{|_{\partial\Omega_i^-(t)}}=0$, we get that
$\displaystyle  \frac{\dd}{\dd t}\int_{\Omega}\varrho_i^-\dd x \leq 0,$
where $\varrho_i^-\geq 0$ is the negative part of $\varrho_i$ ($\varrho_i=\varrho_i^+-\varrho_i^-$). In consequence,
$$\int_\Omega \varrho_i^-(t,x)\dd x \leq \int_\Omega \varrho_i^-(0,x)\dd x = 0,$$
which yields $\varrho_i^-(t,\cdot)\equiv 0$. 

Note that we are able to perform the above calculations provided that the set $\Omega_i^-(t)$ is of class $C^1$. However, from the implicit function theorem and Sard's theorem \cite{sard} it follows that we can find $\delta_n>0$, $\delta_n\to 0$ with $n\to\infty$, such that $\Omega_{i,\delta_n}^-(t):=\{x\in\Omega: \varrho_i(t,x)<\delta_n\}$ have the desired regularity. Then the assertion follows by taking $n\to\infty$.

%\subsection{Energy estimates}\label{energy_sect}

% Now we can proceed with the main estimate. 
% We assume for simplicity that $N=3$ and that we deal with a reaction of the type
%  \[ A + B \rightarrow C, \]
%  however from the proof it is clear that the reasoning can be adjusted to a one-way reaction with arbitrary number of components. In this special case, $\omega_i = -\alpha_i\varrho_1\varrho_2$ for $i=1,2$ and $\omega_3 = (\alpha_1+\alpha_2)\varrho_1\varrho_2$, where $\alpha_i>0$ depend on the reaction rate and molecular weights of $A$ and $B$ respectively. The above example can be easily generalized for the case
%  \[ A_1 + ...+ A_K \rightarrow C_1+ ... + C_L, \]
%  with suitable definitions of $\omega_i$'s.
 
%  \begin{lem}\label{oszac2}
%  For all $i=1,\dots, N$ and some constant $C=C(T,\vec\varrho_0)$ we have
%  \[ \left\|\nabla\left(\varrho_i^{\gamma-1/2}\right)\right\|_{L^2([0,T]\times\Omega)} \leq C. \]
%  \end{lem}

 \smallskip

 {\bf Step II.} In this part we construct the basic energy estimate, in terms of the pointwise bound on the density and the production rates $\omega_i$'s. The goal is to show that
%  \begin{lem}
%  A sufficiently smooth solution to (\ref{diff_react}) satisfies 
 \begin{multline}\label{energy} \int_0^T\int_\Omega 
\mu|\mathbb{D} u|^2 +\lambda(\ddiv u)^2\;\dd x\dd t + \sup_{t\in [0,T]}\sum_{i=1}^N\frac{1}{\gamma_i-1}\int_\Omega p_i(\varrho_i)\dd x + \sum_{i=1}^N \int_0^T\int_\Omega \frac{1}{\varrho_i}|F_i|^2 \dd x\dd t \\
\leq \sum_{i=1}^N\frac{1}{\gamma_i-1}\int_\Omega p_i(\varrho_{0,i})\dd x + C\|\varrho\|_{L^\infty([0,T]\times\Omega)}^{\gamma_3-1}. \end{multline}
%  \end{lem}
%  \begin{proof}
Note that in the general case (\ref{irreversible}), on the right hand side instead of $\gamma_3$ we need to put $\gamma_S$, which is the maximum of exponents corresponding to $C_1,\dots,C_L$.

 Testing the momentum equation by $u$, we get
 \[ \mu\int_\Omega |\mathbb{D} u|^2\dd x + \mu f\int_{\partial\Omega}|u|^2\dd S + \lambda\int_\Omega (\ddiv u)^2\dd x + \int_\Omega \nabla p\cdot u =0. \]
 To obtain the desired estimate, there is a  need to deal with the last term. We have
 \begin{multline*} \int_\Omega \nabla p\cdot u \;\dd x =  \frac{\dd}{\dd t}\sum_{i=1}^N\frac{1}{\gamma_i-1}\frac{1}{m_i}\int_\Omega \varrho_i^{\gamma_i} \;\dd x \\
 +\sum_{i=1}^N\frac{\gamma_i}{\gamma_i-1}\frac{1}{m_i}\int_\Omega\nabla\varrho_i^{\gamma_i-1}\cdot F_i\;\dd x -\sum_{i=1}^N\frac{\gamma_i}{\gamma_i-1}\frac{1}{m_i}\int_\Omega \varrho_i^{\gamma_i-1}\omega_i \;\dd x.  \end{multline*}
 The choice of the fluxes $F_i$ leads to the following important relation:
 \begin{equation}\label{F_i}
\begin{aligned} \sum_{i=1}^N  \frac{\gamma_i}{m_i}\nabla\varrho_i^{\gamma_i-1}\cdot F_i &=
\sum_{i=1}^N \frac{1}{\varrho_i}\left(\frac{1}{m_i}\nabla\varrho_i^{\gamma_i}-\frac{\varrho_i}{\varrho}\nabla p\right)\cdot F_i
= \sum_{i=1}^N \frac{1}{\varrho_i}|F_i|^2.
\end{aligned} \end{equation}
 Moreover, as $\omega_1$ and $\omega_2$ are nonpositive, we have
 \[ -2\sum_{i=1}^N\frac{\gamma_i}{m_i}\int_\Omega \varrho_i^{\gamma_i-1}\omega_i \;\dd x \geq -2\frac{\gamma_3}{m_3}\int_\Omega \varrho_3^{\gamma_3-1}\omega_3 \;\dd x. \]
 In conclusion, our energy estimates read
 \begin{multline*}
 \mu\int_\Omega |\mathbb{D} u|^2 \;\dd x +\mu f\int_{\partial\Omega} |u|^2\;\dd S + \lambda\int_\Omega (\ddiv u)^2 \;\dd x + \frac{\dd}{\dd t}\sum_{i=1}^N\frac{1}{\gamma_i-1}\int_\Omega p_i(\varrho_i)\;\dd x \\
 +\sum_{i=1}^N \int_\Omega \frac{1}{\varrho_i}|F_i|^2 \dd x - \frac{\gamma_3}{m_3}\int_\Omega \varrho_3\omega_3(\vec\varrho\,) \;\dd x \leq 0 \end{multline*}
 and after integrating over time and estimating the last term by 
 \[ C\|\varrho\|_{L^\infty([0,T]\times\Omega)}^{\gamma_3-1}\int_0^T\int_\Omega \omega_3(\vec\varrho\,)\dd x\dd t, \]
 we get the inequality \begin{multline*} \int_0^T\int_\Omega 
\mu|\mathbb{D} u|^2 + \lambda(\ddiv u)^2\;\dd x\dd t + \sup_{t\in [0,T]}\sum_{i=1}^N\frac{1}{\gamma_i-1}\int_\Omega p_i(\varrho_i)\dd x + \sum_{i=1}^N \int_0^T\int_\Omega \frac{1}{\varrho_i}|F_i|^2 \dd x\dd t \\
\leq \sum_{i=1}^N\frac{1}{\gamma_i-1}\int_\Omega p_i(\varrho_{0,i})\dd x + C\|\varrho\|_{L^\infty([0,T]\times\Omega)}^{\gamma_3-1}\int_0^T\int_\Omega \omega_3(\vec\varrho\,)\dd x\dd t. \end{multline*}
However, the straightforward integration of the equation on $\rho_3$ yields
 \[ \int_0^T\int_\Omega \omega_3(\varrho)\;\dd x\dd s = \int_\Omega \varrho_3(T,x)\dd x - \int_\Omega \varrho_{0,3}(x)\dd x \leq C \]
 and the inequality (\ref{energy}) follows.
%  \end{proof}

\smallskip 
 
{\bf Step III.}
 To obtain the estimate for the density, it is required to study the effective viscous flux
 \begin{equation}\label{d} (2\mu+\lambda)\ddiv u = p(\vec\varrho\,)-\frac{1}{|\Omega|}\int_\Omega p(\vec\varrho\,)\dd x + d, \end{equation}
where $\Delta d=0$ in $\Omega$. 
% and $\nabla d=\mathcal{P}\Delta\mathcal{P} u$. 
The formula (\ref{d}) is derived by a projection on the potential part of the momentum equation (\ref{diff_react}).
% and the operator $\mathcal{P}$ is the Helmoltz projection defined below.

% ???The form of the estimates coming from the energy bound together with the pointwise estimate for the density and divergence implies that the velocity is expected to be in the $L^\infty(0,T;W^1_p)$ space for any $p<\infty$. Indeed we expect to have $\nabla u \in L^\infty(0,T;BMO)$. ???
The structure of the slip conditions allows to split the field into two parts:
the divergence-free and potential ones, denoted by $\cP u$ and $\cQ u $ respectively. Altogether we have $u=\cP u + \cQ u$, where
\begin{equation}
    \begin{array}{l}
 \rot \, \cP u = \alpha,\quad \div \, \cP u =0,   \\
%\div \, \cP u =0, \\
n \cdot \cP u=0,
    \end{array}
    \qquad
    \begin{array}{ll}
   \div \, \cQ u = \div\, u,\quad \rot \, \cQ u =0 \quad & \mbox{in }
   \Omega,\\
   %\rot \, \cQ u =0,\\
n\cdot \cQ u=0 & \mbox{at } \partial \Omega.
    \end{array}
\end{equation}
The rotation $\alpha = \rot \,u$ can be determined separately from the equation on the vorticity:
\begin{equation}\label{alf1}
\begin{array}{ll}
    -\mu \Delta \alpha =0 & \mbox{in } \Omega,\\
    \alpha \cdot \tau_1=(2\chi_1 -f/\mu) u \cdot \tau_1,\quad
    \alpha \cdot \tau_2=(f/\mu -2\chi_2) u\cdot \tau_2 & \mbox{at } \partial \Omega,\\
    \frac{\partial(\alpha \cdot n)}{\partial n}=-(\alpha\cdot \tau_1)_{\tau_1} - (\alpha\cdot \tau_2)_{\tau_2}
    & \mbox{at } \partial \Omega.\\
    \end{array}
\end{equation}
The form of boundary conditions in (\ref{alf1}) is a consequence of fine properties of 
slip boundary conditions, described in more details, e.g. in \cite{Mucha, Mucha-Rautmann}.

The theory of maximal regularity estimate for the elliptic problems yields
\begin{equation}
    \|\alpha \|_{L^\infty(0,T;W^1_p)}\leq C \|u|_{\partial \Omega}\|_{L^\infty(0,T;W^{1-1/p}_p)},
\end{equation}
see details in \cite{Mucha-Pokorny}. In consequence, as $\|\mathcal{P}u\|_{L^\infty(0,T;W^{2,p})}\leq C\|\mathrm{rot}u\|_{L^\infty(0,T;W^{1,p})}$, we conclude that
\begin{equation}\label{Pu}
    \|\cP u\|_{L^\infty(0,T;W^{2,p})} \leq C\|u\|_{L^\infty(0,T;W^{1,p})}.
\end{equation}
We will now estimate the harmonic function $d$ from the equation (\ref{d}). We have $\nabla d =\cP \Delta \cP u$ and $\int d(x,t) \;\dd x=0$. Then from (\ref{Pu}) we get
\begin{equation}\label{d_est} \|d\|_{L^\infty(0,T;W^{1,p})}\leq C\|\mathcal{P}u\|_{L^\infty(0,T;W^{2,p})} \leq C\|u\|_{L^\infty(0,T;W^{1,p})}. \end{equation}
% We have $\cP_H \Delta \cP u=\nabla d$ and $\|\mathcal{Q}u(t,\cdot)\|_{W^{1,p}}\leq C\|\ddiv u(t,\cdot)\|_{L^p}$, therefore from (\ref{Pu})
% \begin{equation}\label{d} \|d\|_{L^\infty(0,T;W^{1,p})}\leq C\|\mathcal{P}u\|_{L^\infty(0,T;W^{2,p})} \leq C\|u\|_{L^\infty(0,T;L^q)} + C\|\ddiv u\|_{L^\infty(0,T;L^p)}. \end{equation}

% Combining the above estimate with the Lemmas \ref{lem_interpolacja} and \ref{d_negative}, we obtain

% \begin{equation}\label{d_est}
%     \|d \|_{L^\infty(0,T;L^p)} \leq \varepsilon \|\div u\|_{L^\infty(0,T;L^p)} + C\|p\|_{L^\infty(0,T;L^1)}.
% \end{equation}

%  \subsection{The $L^\infty$ estimates}\label{bounded_section}
 
%  We will now combine the energy estimate and the bounds on $d$ to obtain the $L^\infty$ bound on $\varrho$. The main result in this section states
%  \begin{lem}\label{infty_lem}
%   We have
%   \[ \|\varrho\|_{L^\infty([0,T]\times\Omega)} + \|\ddiv u\|_{L^\infty([0,T]\times\Omega)}\leq C, \]
%   where the constant $C$ depends only on $\vec\varrho_0$.
%  \end{lem}

% \begin{proof} 

% First, note that the estimate (\ref{energy}) allows us to estimate the $L^p$ norms of the pressure by $\|\varrho\|_{L^\infty((0,T)\times\Omega)}$. 

Using (\ref{energy}), we are able to estimate different norms of $p$ in terms of $\|\varrho\|_{L^\infty((0,T)\times\Omega)}$. In particular, we have
 \begin{equation}\label{p_int} \sup_{t\in[0,T]}\int_\Omega p(\vec\varrho\,)\;\dd x \leq C + C\|\varrho\|_{L^\infty([0,T]\times\Omega)}^{\gamma_3-1}. \end{equation}
 
% \[ \|\nabla u\|_{L^\infty(0,T;L^2)}\leq C\|p(\varrho)\|_{L^\infty(0,T;L^2)} \leq C\|\varrho\|_{L^\infty([0,T]\times\Omega)}^{\gamma-1/2}. \]

% Note that if (\ref{gammas}) is satisfied, then there exists $p>3$ such that 
% \[ \tilde\gamma := \gamma_{\max}-\frac{\gamma_{\max}-\gamma_3+1}{p} < \gamma_{\min}. \]
\noindent Moreover, the structure of the pressure implies that
\begin{equation}\label{p_rho} C\varrho^{\gamma_{\min}}-C\leq p(\vec\varrho\,) \leq C\varrho^{\gamma_{\max}}+C, \end{equation}
for some $C>0$ depending on $N,m_1,\dots,m_N$, where $\gamma_{\min}$ and $\gamma_{\max}$ are respectively the minimum and maximum of $\{\gamma_1,\dots,\gamma_N\}$. In consequence, since by the interpolation of $L^p$ spaces
\[ \int_\Omega |p(\vec\varrho\,)|^p\dd x \leq \|p(\vec\varrho\,)\|_{L^\infty(\Omega)}^{p-1}\int_\Omega p(\vec\varrho\,)\dd x \leq C(\|\varrho\|_{L^\infty(\Omega)}^{\gamma_{\max}(p-1)}+1)(1+\|\varrho\|_{L^\infty([0,T]\times\Omega)}^{\gamma_3-1}), \]
we get that 
% \[ \|p(\vec\varrho\,)\|_{L^\infty(0,T;L^p)}\leq C\left(\|\varrho\|_{L^\infty([0,T]\times\Omega)}^{2-\frac{1}{p}}+\|\varrho\|_{L^\infty((0,T)\times\Omega)}^{2-\frac{2}{p}}\right)\leq C+C\|\varrho\|_{L^\infty((0,T)\times\Omega)}^{2-\frac{1}{p}}. \]
\[ \|p(\vec\varrho\,)\|_{L^\infty(0,T;L^p)} \leq C+C\|\varrho\|_{L^\infty((0,T)\times\Omega)}^{\tilde\gamma} \]
for $\tilde\gamma = \gamma_{\max}-\frac{\gamma_{\max}-\gamma_3+1}{p}$.
Therefore from the elliptic estimates we also have
% \[ \|u\|_{L^\infty(0,T;W^{1,p})} \leq C+C\|\varrho\|_{L^\infty((0,T)\times\Omega)}^{2-\frac{1}{p}} \]
\[ \|u\|_{L^\infty(0,T;W^{1,p})} \leq C+C\|\varrho\|_{L^\infty((0,T)\times\Omega)}^{\tilde\gamma} \]
and if $p>3$, then using (\ref{d_est}) we get
% \begin{equation}\label{d_infty} \|d\|_{L^\infty((0,T)\times\Omega)} \leq C\|d\|_{L^\infty(0,T;W^{1,p})} \leq C+C\|\varrho\|_{L^\infty((0,T)\times\Omega)}^{2-\frac{1}{p}}. \end{equation}
\begin{equation}\label{d_infty} \|d\|_{L^\infty((0,T)\times\Omega)} \leq C\|d\|_{L^\infty(0,T;W^{1,p})} \leq C+C\|\varrho\|_{L^\infty((0,T)\times\Omega)}^{\tilde\gamma}. \end{equation}

Note that if $\gamma_1,\dots,\gamma_N$ satisfy the relation (\ref{gammas}), then there exists such $p>3$ that $\tilde\gamma < \gamma_{\min}$. Moreover, as $\gamma_{\min}\leq \gamma_3 \leq \gamma_{\max}$, it also follows that $\gamma_3-1<\gamma_{\min}$.

\smallskip

{\bf Step IV.} The last step finishes our proof of Lemma \ref{oszac1}. The key point of the estimate requires a special point-wise conditional estimate controlling the sign of the divergence of "bad" regions.

\begin{lem}\label{divu_positive}
    If $\varrho(t,x)>\|\varrho\|_{L^\infty((0,T)\times\Omega)}-1$ and $\|\varrho\|_{L^\infty((0,T)\times\Omega)}$ is sufficiently large, then 
    \[ \ddiv u(t,x) >0. \]
\end{lem}
\begin{proof}
\noindent
Combining the estimates (\ref{p_int}), (\ref{p_rho}) and (\ref{d_infty}) with (\ref{d}), we obtain
\[\begin{aligned} \ddiv u(t,x) &= p(\vec\varrho(t,x))-\frac{1}{|\Omega|}\int_\Omega p(\vec\varrho\,)\dd x -d \\
&\geq C\varrho(t,x)^{\gamma_{\min}}-C - C\|\varrho\|_{L^\infty((0,T)\times\Omega)}^{\gamma_3-1}-\|d\|_{L^\infty((0,T)\times\Omega)}^{\tilde\gamma} \\
&\geq C\left(\left(\|\varrho\|_{L^\infty((0,T)\times\Omega)}-1\right)^{\gamma_{\min}} - \|\varrho\|_{L^\infty((0,T)\times\Omega)}^{\tilde\gamma}-\|\varrho\|_{L^\infty((0,T)\times\Omega)}^{\gamma_3-1}-1\right). 
\end{aligned}\]
As $\gamma_{\min}>\max(\tilde\gamma,\gamma_3-1)$, the function $ z\mapsto (z-1)^{\gamma_{\min}}-z^{\tilde\gamma}-z^{\gamma_3-1}-1$
is strictly positive for sufficiently large $z$ and thus we arrive to the desired conclusion.
\end{proof}

\smallskip 

To finish the proof of Lemma \ref{oszac1}, let us assume that $\|\varrho\|_{L^\infty((0,T)\times\Omega)}>\|\varrho_0\|_{L^\infty(\Omega)}+1$ (otherwise $C$ from the statement of the Lemma is just equal to $\|\varrho_0\|_{L^\infty(\Omega)}+1$) and let $k>\|\varrho\|_{L^\infty((0,T)\times\Omega)}-1$. We test the continuity equation on $\varrho$ by $(\varrho-k)_+$, where $f_+$ denotes the positive part of $f$. Then 
\[\begin{aligned} \frac{1}{2}\int_\Omega (\varrho-k)_+^2(t,\cdot)\;\dd x - \frac{1}{2}\int_\Omega (\varrho_0-k)_+^2\;\dd x &=  \int_0^t\int_\Omega \varrho u\cdot\nabla(\varrho-k)_+\;\dd x\dd s \\
&= -\int_0^t\int_\Omega \ddiv u\left(\frac{1}{2}(\varrho-k)_+^2 + k(\varrho-k)_+\right)\;\dd x\dd s.
\end{aligned}\]

Therefore using the fact that $\varrho_0<k$, we obtain
\[
\sup_{\tau \in [0,t)}\int_\Omega (\varrho-k)_+^2(x,\tau) \dd x \leq -\int_0^t\int_\Omega (\varrho-k)_+(\varrho+k)\ddiv u\dd x\dd s. \]

Now if $\|\varrho\|_{L^\infty((0,T)\times\Omega)}$ is large enough, then from Lemma \ref{divu_positive} the right hand side is non-positive.  
In consequence $(\varrho-k)_+=0$ a.e. in $(0,T)\times\Omega)$ and thus
\[ \varrho<\|\varrho\|_{L^\infty((0,T)\times\Omega)}-1 \quad \text{a.e. in} \quad [0,T]\times\Omega. \]
This leads to contradiction, and in consequence $\|\varrho\|_{L^\infty((0,T)\times\Omega)}$ is bounded by some constant depending only on $\|\vec\varrho_0\|_{L^\infty(\Omega)}$. Thus Lemma \ref{oszac1} is proved.

\smallskip 
\begin{rem}
As $\|\varrho\|_{L^\infty((0,T)\times\Omega)}\leq C$, by the relation (\ref{d}) $\|\ddiv u\|_{L^\infty((0,T)\times\Omega)}\leq C$ as well. Therefore from the elliptic estimates $\|\nabla\cQ u\|_{L^\infty(0,T;BMO)}\leq C$, and by (\ref{Pu})
\begin{equation}\label{bmo}  
 \|\nabla u\|_{L^\infty(0,T;BMO)}\leq C. 
 \end{equation}
 \end{rem}

\section{Existence for an approximated system}\label{existence_sect}

The goal of this section is to prove Theorem \ref{appr_th}. Similarly as in Section \ref{a_priori_sect}, the proof is valid for different $\gamma_i$'s under the constraint (\ref{gammas}), however for simplicity we restrict themselves to the case $\gamma_1=\dots=\gamma_N=2$.
% \begin{tw}\label{appr_th}
%     For any $T>0$, there exists at least one weak solution $(\varrho_1,\ldots,\varrho_N,u)$ to the system
%     \begin{equation}\label{approx}
%     \begin{aligned}
%     \partial_t\varrho_{i} + \ddiv(\varrho_{i} u) -\ddiv F_{i} &= \omega_{i}(\vec\varrho\,) + \varepsilon\Delta\varrho_i, \quad i=1,\dots,N ,\\
%    -\mu\Delta u -\nabla((\mu+\lambda)\ddiv u) + \nabla p(\vec\varrho\,) &= 0,
%     \end{aligned}
% \end{equation}
% in $[0,T]\times\Omega$ with 
% \[ {\varrho_i}_{|_{t=0}} = \varrho_{0,i,\varepsilon}\in C^\infty(\Omega), \quad \varrho_{0,i,\varepsilon}\to \varrho_{0,i} \quad \text{in} \quad L^2(\Omega), \]
% the Neumann boundary condition $\frac{\dd\varrho_i}{\dd n}=0$ on $\partial\Omega$ 
% and the slip boundary conditions (\ref{slip}) for $u$. Moreover, the obtained solution satisfies
% \[ \varrho_i \geq 0, \quad i=1,\ldots,N, \]
% \[ \|\varrho_i\|_{L^\infty((0,T)\times\Omega)}, \left\|\frac{1}{\sqrt{\varrho_i}}F_i\right\|_{L^2((0,T)\times\Omega)}, \sqrt{\varepsilon}\|\varrho_i\|_{L^2(0,T;H^1)} \leq C, \quad i=1,\ldots,N \]
% and
% \[ \|u\|_{L^2(0,T;H^1)}, \|\nabla u\|_{L^\infty(0,T;BMO)} \leq C, \]
% where the constant $C$ does not depend on $\varepsilon$ or $T$.
% \end{tw}
%
Note that the Neumann condition on $\varrho_i$ implies that $F_{i}\cdot n=0$ for $i=1,\dots,N$ as well, since
\[ F_i\cdot n = 2\varrho_i(\nabla\varrho_i\cdot n) - 2\frac{\varrho_i}{\varrho}\sum_{j=1}^N\varrho_j(\nabla\varrho_j\cdot n) = 0 \quad \text{on} \quad \partial\Omega \]

\smallskip 

\paragraph{Proof of Theorem \ref{appr_th}:} Let us introduce another level of approximation, depending on a small parameter $\delta$. We improve the integrability in time of the densities, by adding the terms $\delta\tilde\varrho^{\beta-2}\varrho_i$ for
$\displaystyle\tilde\varrho = \sum_{i=1}^N |\varrho_i|$ and $\beta= 6$. We also introduce the truncation of $\varrho_i$, defined as 
\[ \varrho_i^\delta := \mathrm{sgn}(\varrho_i)\min\left(|\varrho_i|,\frac{1}{\delta}\right). \]
 Then we put
\[ \omega_i^\delta(\vec\varrho) := \omega_i\Big(\varrho_1^\delta,\dots, \varrho_N^\delta\Big) \mbox{ \ \ 
and \ \ }
 \tilde\varrho_i^\delta := \min\left(|\varrho_i|,\frac{1}{\delta}\right) = |\varrho_i^\delta|. \]
There is also the need to approximate the pressure and the diffusion fluxes. We put
\begin{equation}
    p_{i}^\delta(\varrho_i)=\frac{2}{m_i} \int_0^{\varrho_i}\min\left(|w|,\frac{1}{\delta}\right)\dd w
\end{equation}
and define $F_{i}^\delta$ as
\begin{equation}
    F_{i}^\delta= \nabla p_{i}^\delta(\varrho_i)-
\frac{\tilde\varrho_i^\delta}{\tilde \varrho^\delta}\nabla p^\delta,
%\end{equation}
\mbox{ \ where \ \ }
 \tilde \varrho^\delta = \sum_{k=1}^N \left|\varrho_k^\delta\right| \quad \text{and} \quad p^\delta = \sum_{k=1}^N p_k^\delta(\varrho_k). 
\end{equation}
Note that $F_{i}^\delta$ satisfy the same necessary properties as $F_i$, namely $\displaystyle\sum_{i=1}^N F_{i}^\delta=0$ and
\[ \sum_{i=1}^N \frac{1}{m_i}\nabla\varrho_i\cdot F_{i}^\delta = \sum_{i=1}^N \frac{|F_{i}^\delta|^2}{|\varrho_i^\delta|} \geq 0. \]

Finally, the $(\delta,\varepsilon)$-approximative system reads

\begin{equation}\label{approx2}
    \begin{aligned}
    \partial_t\varrho_{i} + \ddiv(\varrho_{i} u) -\ddiv F_{i}^\delta +\delta \tilde\varrho^{\beta-2}\varrho_i &= \omega_{i}^\delta(\vec\varrho\,) + \varepsilon\Delta\varrho_i, \quad i=1,\dots,N ,\\
   -\mu\Delta u -\nabla((\mu+\lambda)\ddiv u) + \nabla p(\vec\varrho\,) &= 0.
    \end{aligned}
\end{equation}

% \begin{equation}\label{approx}
%     \begin{aligned}
%     \partial_t\varrho_{i} + \ddiv(\tilde\varrho_{i}^M u) -\ddiv F_{i}^M &= \omega_{i}^M(\vec\varrho) + \varepsilon\Delta\varrho_i, \quad i=1,\dots,N ,\\
%    -\mu\Delta u -\nabla((\mu+\lambda)\ddiv u) + \nabla p^M(\vec\varrho) &= 0
%     \end{aligned}
% \end{equation}

% with 
% \[ {\varrho_i}_{|_{t=0}} = \varrho_{0,i,\varepsilon}\in C^\infty(\Omega), \quad \varrho_{0,i,\varepsilon}\to \varrho_{0,i} \quad \text{in} \quad L^2(\Omega), \]
% the Neumann boundary condition $\frac{\dd\varrho_i}{\dd n}=0$ on $\partial\Omega$ 
% and the slip boundary conditions (\ref{slip}) for $u$ (note that from the Neumann condition on the components it also follows that $F_{i}^M\cdot n=0$ for $i=1,\dots,N$).

\subsection{The Galerkin approximation}
We  construct  solutions to (\ref{approx2}) using the Galerkin method. Note that taking into account the slip boundary conditions, the weak formulation of the second equation of (\ref{approx2}) reads
\begin{equation}
\int_0^T\int_\Omega \mu\mathbb{D}u:\nabla\phi + \lambda \div u \,\div \phi \;\dd x\dd t + \int_0^T\int_{\partial\Omega} \mu f \; u\cdot\phi\;\dd S\dd t =
\int_0^T\int_\Omega p \,\div \phi \;\dd x\dd t,
\end{equation}
provided that $\phi \in C^\infty([0,T]\times\overline{\Omega},R^3)$ with $\phi \cdot n=0$ at $\partial \Omega$. We show that

\begin{lem}
    The equation (\ref{approx2}) admits a global weak solution in $[0,T]\times\Omega$.
\end{lem}

\begin{proof} Let $\{w_k\}_{k\in\mathbb{N}}\in C^\infty(\overline{\Omega})$ and $\{v_k\}_{k\in\mathbb{N}}\in (C^\infty(\overline{\Omega}))^d$ be the orthogonal bases of $H^1(\Omega)$ and $(H^1(\Omega))^d$, such that
\[ \frac{\partial w_k}{\partial n} = 0, \quad v_k\cdot n=0 \quad \text{on} \quad \partial\Omega \quad \text{for all} \quad k\in\mathbb{N}, \]
$\{w_k\}$ is orthonormal in $L^2(\Omega)$ and
\[ \forall_{k,l\in\mathbb{N}}: \quad \mu\int_\Omega \mathbb{D} v_k:\mathbb{D} v_l \;\dd x + \mu f\int_{\partial\Omega} v_k\cdot v_l\;\dd S + \lambda\int_\Omega \ddiv v_k\ddiv v_l\;\dd x = \delta_{k,l}. \]
Define
\[ \varrho_{i,n}=\sum_{k=1}^n a_{i,k}(t)w_k \quad
\text{and} \quad u_n = \sum_{k=1}^n b_k(t)v_k, \]
where the coefficients $a_{i,k}$ satisfy the system of $n\times N$ ODEs
\begin{multline}\label{a} \dot a_{i,k} = \sum_{l,m=1}^n a_{i,l}b_m\int_\Omega w_l v_m\nabla w_k \dd x - \int_\Omega F_{i,n}^\delta\cdot\nabla w_k\dd x \\
- \varepsilon a_{i,k}\int_\Omega |\nabla w_k|^2\dd x - \delta\sum_{l=1}^n a_{i,l}\int_\Omega \tilde\varrho_n^{\beta-2}w_lw_k\dd x + \int_\Omega \omega_i^\delta(\vec\varrho_n) w_k \dd x \end{multline}
and $b_k$ is given by
\begin{equation}\label{b} b_k = \int_\Omega p(\vec\varrho_n)\ddiv v_k \dd x \end{equation}
for $\vec\varrho_n = (\varrho_{1,n},\dots,\varrho_{N,n})$. Since the functions
\[ (z_1,\dots,z_N)\mapsto \frac{z_iz_j}{\sum_{k=1}^N z_k}, \quad i,j=1,\dots,N \]
are Lipschitz continuous on $\{z\in\mathbb{R}^N: z_k\geq 0 \;\; \forall_{k=1,\dots,N}\}$, the terms $\displaystyle\int_\Omega F_i^\delta\cdot\nabla w_k\dd x$ are locally Lipschitz with respect to $a_{i,k}\in\mathbb{R}^{N\times n}$. Therefore, plugging the relation (\ref{b}) into (\ref{a}), we see that the right hand side of (\ref{a}) is locally Lipschitz and from the Picard-Lindel\"of theorem there exists a local in time solution to (\ref{a}), corresponding to the initial condition
\[ a_{i,k}(0) = \langle\varrho_{0,i,\varepsilon},w_k\rangle_{L^2}. \]

To obtain the solution on a whole interval $[0,T]$, we need to find the global estimate. We test the momentum equation by $u_n$ and integrate by parts. Then, we get

\[\begin{aligned} \mu\int_0^T\int_\Omega |\mathbb{D} u_n|^2\dd x\dd t + \mu f\int_0^T\int_{\partial\Omega} |u_n|^2\;\dd S\dd t + \sup_{t\in [0,T]}\sum_{i=1}^N\frac{1}{m_i}\int_\Omega\varrho_{i,n}^2\dd x \\
+ \sum_{i=1}^N\int_0^T\int_\Omega \frac{|F_{i}^\delta|^2}{\tilde\varrho_{i,n}^\delta}\dd x\dd t +2\delta\sum_{i=1}^N\frac{1}{m_i}\int_0^T\int_\Omega \tilde\varrho_n^{\beta-2}\varrho_{i,n}^2\dd x\dd t \\
+2\varepsilon\sum_{i=1}^N\frac{1}{m_i}\int_0^T\int_\Omega |\nabla\varrho_{i,n}|^2\dd x\dd t - 2\sum_{i=1}^N\frac{1}{m_i}\int_0^T\int_\Omega\varrho_{i,n}\omega_i^\delta(\vec\varrho_n)\dd x\dd t &\leq \sum_{i=1}^N\frac{1}{m_i}\int_\Omega\varrho_{0,i,\varepsilon}^2\dd x. 
\end{aligned}\]
From Korn's and Poincar\'e's inequalities (see e. g. \cite{novotny-straskraba}), we have
\[ \|u_n\|_{W^{1,2}(\Omega)}^2 \leq C\int_\Omega|\mathbb{D}u_n|^2\;\dd x + C\int_{\partial\Omega} |u_n|^2\;\dd S. \]
Moreover, as $\omega_i^\delta\in L^\infty((0,T)\times\Omega)$, from the Cauchy inequality we have
\[ \int_0^T\int_\Omega \varrho_{i,n}\omega_i^\delta\dd x \leq \eta\|\varrho_{i,n}\|_{L^\infty(0,T;L^2)}^2 + \frac{CT^2}{\eta}\|\omega_i^\delta\|_{L^\infty((0,T)\times\Omega)}^2. \]
Therefore choosing $\eta$ sufficiently small, we get the estimates

\begin{multline}\label{ap_energy} \|u_n\|_{L^2(0,T;W^{1,2})}^2 + \sum_{i=1}^N\left(\|\varrho_{i,n}\|_{L^\infty(0,T;L^2)}^2 + \varepsilon\|\nabla\varrho_{i,n}\|_{L^2((0,T)\times\Omega)}^2+\delta\|\varrho_{i,n}\|_{L^\beta((0,T)\times\Omega)}^\beta\right) \\
\leq C + C(\delta)T. \end{multline}
In particular, we get the bound on $\sum_{i=1}^N\sum_{k=1}^n a_{i,k}^2(t)$, which provides that we can extend the solution to a whole interval $[0,T]$. The inequality (\ref{ap_energy}) also provides the estimates uniform in $n$, which allow us to extract weakly convergent subsequences (indexed again by $n$)
\[ u_n,\varrho_{1,n},\dots,\varrho_{N,n} \rightharpoonup u,\varrho_1,\dots,\varrho_N \quad \text{in} \quad L^2(0,T;W^{1,2}). \]
From the estimate on $\|\varrho_{i,n}\|_{L^\beta((0,T)\times\Omega)}$ it follows that $\|p_n\|_{L^{\beta/2}((0,T)\times\Omega)} \leq C$. Therefore, by  the momentum equation, we deduce also 
\[ \|u_n\|_{L^{\beta/2}(0,T;W^{1,\beta/2})} \leq C. \]
Since $\beta=6$, it follows that 
\begin{equation}\label{beta_est} \|\varrho_{i,n}u_n\|_{L^2(0,T;L^q)} \leq C \quad \text{for any} \quad q<6 \end{equation}
and in particular
\[ \|\varrho_{i,n} u_n\|_{L^2((0,T\times\Omega)}, \|F_{i,n}^\delta\|_{L^2((0,T)\times\Omega)} \leq C. \]
In consequence, we obtain the uniform bound on $\|\partial_t\varrho_{i,n}\|_{L^2(0,T;H^{-1})}$ and from the Aubin-Lions Lemma
\[ \varrho_{i,n}\to \varrho_i \quad \text{in} \quad L^2((0,T)\times\Omega). \]
This allows us to pass to the limit with $n\to\infty$ in the weak formulation of (\ref{approx2}) and in consequence obtain a weak solution.
\end{proof}

\subsection{Nonnegativity of the components}
We will now prove that for $(\varrho_1,\dots,\varrho_N,u)$ solving (\ref{approx2}), we have
\[ \varrho_i \geq 0 \quad \text{a. e.} \quad \forall_{i=1,\dots,N}. \]
Note that the low regularity of $\varrho_i$ does not allow us to replicate the argument from Section \ref{a_priori_sect}. Instead, we will obtain the result by choosing a suitable test function. 

Let 
\[ f_h(x)=\left\{\begin{aligned}
    h\ln h, \quad x>0, \\
    x+h\ln(|x|+h), \quad x\leq 0.
\end{aligned}\right.\]
We test the equation on $\varrho_i$ by $f_h'(\varrho_i)=\frac{|\varrho_i|}{|\varrho_i|+h}\mathbbm{1}_{\{\varrho_i\leq 0\}}=\frac{\varrho_i^-}{\varrho_i^-+h}$. Since $f_h''$ is bounded and $\varrho_i$ is in $L^2(0,T;H^1)$, $f_h'(\varrho_i)\in L^2(0,T;H^1)$ as well. First, observe that if $\varphi_\eta(t,x)$ is the standard mollifier over time and space, then for $(\varrho_i)_\eta=\varrho_i\ast\varphi_\eta$ we have $f_h'((\varrho_i)_\eta)\to f_h'(\varrho_i)$ in $L^2(0,T;H^1)$ and $\partial_t(\varrho_i)_\eta\rightharpoonup^*\partial_t\varrho_i$ in $L^2(0,T;H^{-1})$. Thus
\[\begin{aligned} \int_0^t \langle\partial_t\varrho_i, f_h'(\varrho_i)\rangle \;\dd s &= \lim_{\eta\to 0}\int_0^t\int_\Omega \partial_t(\varrho_i)_\eta f_h'((\varrho_i)_\eta))\;\dd x\dd s \\
&= \lim_{\eta\to 0}\int_0^t\int_\Omega \partial_t f_h((\varrho_i)_\eta)\;\dd x\dd s \\
&= \lim_{\eta\to 0}\left[\int_\Omega f_h((\varrho_i)_\eta(t,\cdot))\;\dd x - \int_\Omega f_h((\varrho_i)_\eta(0,\cdot))\;\dd x\right] \\
&= \int_\Omega f_h(\varrho_i(t,\cdot))\;\dd x - \int_\Omega f_h(\varrho_i(0,\cdot))\;\dd x,
\end{aligned}\]
where in the last line we used the fact that $\varrho_i\in C(0,T;L^2)$. Employing the fact that $\varrho_i^-(0,x)=0$ we therefore get
\[\begin{aligned} \int_0^t\langle \partial_t\varrho_i, f_h'(\varrho_i)\rangle &= \int_\Omega \Big[(\varrho_i+h\ln(|\varrho_i|+h))\mathbbm{1}_{\varrho_i\leq 0} + h\ln h\mathbbm{1}_{\varrho_i>0}\Big] \;\dd x - h\ln h |\Omega| \\
&= -\int_\Omega \varrho_i^- \;\dd x + \int_\Omega h\ln(\varrho_i^-+h)\;\dd x - h\ln h|\Omega|.
\end{aligned}\]
Moreover, for the fluxes we get
\[\begin{aligned}
    -\int_0^t\int_\Omega \ddiv F_i^\delta\,f_h'(\varrho_i)\;\dd x\dd s =& \int_0^t\int_\Omega \left(\frac{1}{m_i}\tilde\varrho_i^\delta\nabla\varrho_i-\frac{\tilde\varrho_i^\delta}{\tilde\varrho^\delta}\nabla p^\delta\right)f_h''(\varrho_i)\nabla\varrho_i \;\dd x\dd s \\
    =& -\int_0^t\int_\Omega \frac{1}{m_i}\frac{h}{(|\varrho_i|+h)^2}\min(|\varrho_i|,\frac{1}{\delta})|\nabla\varrho_i|^2\mathbbm{1}_{\varrho_i\leq 0} \;\dd x\dd s \\
    &+ \int_0^t\int_\Omega \frac{h}{(|\varrho_i|+h)^2}\frac{\min(|\varrho_i|,\frac{1}{\delta})}{\tilde\varrho^\delta}\nabla p^\delta\nabla\varrho_i\mathbbm{1}_{\varrho_i\leq 0} \;\dd x\dd s \\
    =& -\int_0^t\int_\Omega \frac{1}{m_i}\frac{h}{(\varrho_i^-+h)^2}\min(\varrho_i^-,\frac{1}{\delta})|\nabla\varrho_i^-|^2\;\dd x\dd s \\
    &+ \int_0^t\int_\Omega \frac{h}{(\varrho_i^-+h)^2}\frac{\min(\varrho_i^-,\frac{1}{\delta})}{\tilde\varrho^\delta}\nabla p^\delta\cdot\nabla\varrho_i^-\;\dd x\dd s
\end{aligned}\]
Combining the above calculations with the remaining terms, we get
\[\begin{aligned}
    -\int_\Omega \varrho_i^-\;\dd x + \int_\Omega (h\ln(\varrho_i^-+h)-h\ln h)\;\dd x \\
    -\varepsilon\int_0^t\int_\Omega \frac{h}{(\varrho_i^-+h)^2}|\nabla\varrho_i^-|^2\;\dd x\dd s -\delta\int_0^t\int_\Omega \tilde\varrho^{\beta-2}\frac{(\varrho_i^-)^2}{\varrho_i^-+h}\;\dd x\dd s \\
    +\int_0^t\int_\Omega \frac{h\varrho_i^-}{(\varrho_i^-+h)^2}u\cdot\nabla\varrho_i^- \;\dd x\dd s \\
    -\int_0^t\int_\Omega \frac{1}{m_i}\frac{h}{(|\varrho_i|+h)^2}\min(\varrho_i^-,\frac{1}{\delta})|\nabla\varrho_i^-|^2\;\dd x\dd s \\
    -\int_0^t\int_\Omega \frac{h}{(\varrho_i^-+h)^2}\frac{\min(\varrho_i^-,\frac{1}{\delta})}{\tilde\varrho^\delta}\nabla p^\delta\cdot\nabla\varrho_i^-\;\dd x\dd s &= \int_0^t\int_\Omega \omega_i^\delta(\vec\varrho)\frac{\varrho_i^-}{\varrho_i^-+h} \;\dd x\dd s.
\end{aligned}\]
Neglecting the nonpositive terms, we end up with
\begin{equation}\label{nonneg_h}
\begin{aligned}
    -\int_\Omega \varrho_i^-\;\dd x + & \int_\Omega h\ln(\varrho_i^-+h)- h\ln h\;\dd x + \int_0^t\int_\Omega \frac{h\varrho_i^-}{(\varrho_i^-+h)^2}u\cdot\nabla\varrho_i^-\;\dd x\dd s \\
    &-\int_0^t\int_\Omega \frac{h}{(\varrho_i^-+h)^2}\frac{\min(\varrho_i^-,\frac{1}{\delta})}{\tilde\varrho^\delta}\nabla p^\delta\cdot\nabla\varrho_i^-\;\dd x\dd s \geq \int_0^t\int_\Omega \omega_i^\delta(\vec\varrho)\frac{\varrho_i^-}{\varrho_i^-+h}\;\dd x\dd s
\end{aligned}\end{equation}
Since 
\[ \frac{h\,\min(\varrho_i^-,\frac{1}{\delta})}{(\varrho_i^-+h)^2} \leq \frac{h}{\varrho_i^-+h}\cdot\frac{\varrho_i^-}{\varrho_i^-+h}\leq 1, \]
from the dominated convergence theorem we have
\[ \int_0^t\int_\Omega \frac{h\varrho_i^-}{(\varrho_i^-+h)^2}u\cdot\nabla\varrho_i^-\;\dd x\dd s \to 0 \]
and
\[ \int_0^t\int_\Omega \frac{h}{(\varrho_i^-+h)^2}\frac{\min(\varrho_i^-,\frac{1}{\delta})}{\tilde\varrho^\delta}\nabla p^\delta\nabla\varrho_i^- \;\dd x\dd s \to 0 \]
as $h\to 0$. Therefore passing to the limit in (\ref{nonneg_h}), we get
\[ \int_\Omega \varrho_i^-\;\dd x \leq -\int_0^t\int_\Omega \omega_i^\delta(\vec\varrho)\mathbbm{1}_{\varrho_i\leq 0} \;\dd x\dd s \]

From the assumptions on $\omega_i$ and Remark \ref{omega_rem}, for $\varrho_i\leq 0$ we have $\omega_i^\delta\geq 0$. Therefore 
\[ \int_\Omega \varrho_i^-(t,\cdot)\dd x \leq 0. \]
However, $\varrho_i^-\geq 0$ by definition, and thus $\varrho_i^-(t,\cdot)\equiv 0$ for all $t\in [0,T]$.

\subsection{Boundedness of $\varrho$}

It turns out that all components are bounded, and their $L^\infty$ bound does not depend on $\varepsilon$ and $\delta$. As we already know that $\varrho_i\geq 0$ for all $i=1,\dots,N$, it is sufficient to prove the $L^\infty$ estimate for the total density $\varrho=\sum_{i=1}^N\varrho_i$. Since $\varrho=\tilde\varrho$, the total density satisfies the equation
\[ \partial_t\varrho + \ddiv(\varrho u)-\varepsilon\Delta\varrho + \delta\varrho^{\beta-1} = 0. \]

First, let us show that $\varrho\in L^\infty((0,T)\times\Omega)$. Recall the following result (see e.g. Lemmas 7.37-38 in \cite{novotny-straskraba}):

\begin{prop}\label{parabolic_th}
For a bounded $\Omega\in C^{2,\theta}$, $0<\theta\leq 1$, let $h\in L^p(0,T;L^q)$ for $1<p,q<\infty$. Then the solution to
\[ \partial_t\varrho -\varepsilon\Delta\varrho = h, \quad \frac{\partial\varrho}{\partial n}=0 \]
with $\varrho_{|_{t=0}}=\varrho_0$ satisfies the estimate
\begin{equation}\label{parabolic_est} \varepsilon^{1-1/p}\|\varrho\|_{L^\infty(0,T;W^{2-{2/p},q})} + \|\partial_t\varrho\|_{L^p(0,T;L^q)} + \varepsilon\|\varrho\|_{L^p(0,T;W^{2,q})} \leq C \left(\varepsilon^{1-1/p}\|\varrho_0\|_{W^{1,q}} + \|h\|_{L^p(0,T;L^q)}\right). \end{equation}
Moreover, if $h=\ddiv w$, $w\in L^p(0,T;L^q)$ and $p\geq 2$, then
\begin{equation}\label{parabolic_est2} \varepsilon^{1-1/p}\|\varrho\|_{L^\infty(0,T;L^q)} + \varepsilon\|\nabla\varrho\|_{L^p(0,T;L^q)} \leq C\left(\varepsilon^{1-1/p}\|\varrho_0\|_{L^q} + \|w\|_{L^p(0,T;L^q)}\right). \end{equation}
\end{prop}

By virtue of (\ref{beta_est}), $\varrho u\in L^2(0,T;L^q)$ for a suitable $q>3$. Therefore from Proposition \ref{parabolic_th}, we have $\varrho \in L^\infty(0,T;L^q).$
Then, using the property (\ref{p_rho})
from the elliptic estimates we have $p\in L^\infty(0,T;L^{q/2})$ and thus
we get that $u\in L^\infty(0,T;W^{1,q/2}).$
In consequence, 
$\varrho u\in L^\infty(0,T;L^r)$
for some $r>3$. Therefore applying the De Giorgi  technique (see e.g. \cite{wu-yin-wang}, Chapter 4), we get that
\[ \varrho \in L^\infty((0,T)\times\Omega). \]

We  now prove that in fact the $L^\infty$ bound on $\varrho$ depends only on $\|\varrho_0\|_{L^\infty(\Omega)}$, using the reasoning analogous as in Section \ref{a_priori_sect}. Similarly as for the a priori estimates, we have
    \[ -2\sum_{i=1}^N\frac{1}{m_i}\int_\Omega \varrho_i\omega_i^\delta \;\dd x \geq -2\frac{1}{m_3}\int_\Omega \varrho_3\omega_3^\delta \;\dd x. \]
    Therefore analogously we get the bound

\begin{equation}\label{energy2} \begin{aligned}
\mu\int_0^T\int_\Omega |\nabla u|^2\dd x\dd t +& \sup_{t\in [0,T]}\sum_{i=1}^N\frac{1}{m_i}\int_\Omega\varrho_{i}^2\dd x + \sum_{i=1}^N\int_0^T\int_\Omega \frac{|F_{i}^\delta|^2}{\tilde\varrho_{i}^\delta}\dd x\dd t \\ 
+ 2\delta\sum_{i=1}^N\frac{1}{m_i}\int_0^T\int_\Omega \varrho^{\beta-2}\varrho_{i}^2\dd x\dd t
&+ \varepsilon\sum_{i=1}^N\frac{2}{m_i}\int_0^T\int_\Omega |\nabla\varrho_{i}|^2\dd x\dd t \\
& \leq \sum_{i=1}^N\frac{1}{m_i}\int_\Omega\varrho_{0,i,\varepsilon}^2\dd x + C\|\varrho_3\|_{L^\infty((0,T)\times\Omega)}\int_0^T\int_\Omega \omega_3^\delta\dd x. 
\end{aligned}\end{equation}
As 
\[ \int_\Omega \varrho(t,x) \dd x + \delta\int_0^t\int_\Omega \varrho^{\beta-1}\dd x = \int_\Omega \varrho_{0,\varepsilon}\dd x, \]
we also have
 \begin{equation}\label{omega_est} \int_0^T\int_\Omega \omega_3^\delta\dd x\dd t = \int_\Omega\varrho_i(T,\cdot)\dd x - \int_\Omega \varrho_{i,0,\varepsilon}\dd x + \delta\int_0^T\int_\Omega \varrho^{\beta-2}\varrho_i\dd \dd s \leq \int_\Omega\varrho_{0,\varepsilon}\dd x \leq \int_\Omega \varrho_0\dd x. \end{equation}

Now we repeat Step IV of the proof of Lemma \ref{oszac1}. If $\|\varrho\|_{L^\infty((0,T)\times\Omega)}\leq \|\varrho_0\|_{L^\infty(\Omega)}+1$, then the assertion follows. If $\|\varrho\|_{L^\infty((0,T)\times\Omega)}>\|\varrho_0\|_{L^\infty(\Omega)}+1$, then let
$k>\|\varrho\|_{L^\infty((0,T)\times\Omega)}-1$. We test the equation on $\varrho$ by $(\varrho-k)_+$. We have
\[\begin{aligned} \frac{1}{2}\int_\Omega (\varrho-k)_+^2(t,x)\dd x - \frac{1}{2}\int_\Omega (\varrho_0-k)_+^2\dd x \\
+ \varepsilon\int_0^t\int_\Omega |\nabla(\varrho-k)_+|^2\dd x\dd s +\delta\int_0^t\int_\Omega \varrho^{\beta-1}(\varrho-k)_+\dd x\dd s &= \int_0^t\int_\Omega \varrho u\cdot\nabla(\varrho-k)_+\dd x\dd s \\
&= -\int_0^t\int_\Omega \ddiv u\left(\frac{1}{2}(\varrho-k)_+^2 + k(\varrho-k)_+\right)\dd x\dd s.
\end{aligned}\]
Therefore using the fact that $\varrho_0<k$, we obtain
\[ \int_\Omega (\varrho-k)_+^2\dd x \leq -\int_0^t\int_\Omega (\varrho-k)_+(\varrho+k)\ddiv u\dd x\dd s. \]

Combining (\ref{energy2}) and (\ref{omega_est}), we obtain the estimate
\[ \|p\|_{L^\infty(0,T;L^1)} \leq C+C\|\varrho\|_{L^\infty((0,T)\times\Omega)}. \]
Then we proceed in the same way as in the proof of Lemma \ref{divu_positive} to get $\div u\mathbbm{1}_{\varrho>k}> 0$ for $k$ sufficiently large and in consequence
\[ \|\varrho\|_{L^\infty((0,T)\times\Omega)}\leq C,\]
where $C$ is the absolute constant independent from $\varepsilon$ and $\delta$, given in terms of the $L^\infty$
norm of $\rho_0$.

\subsection{Limit passage with $\delta\to 0$.}
Having the uniform $L^\infty$ bounds on $\varrho$ and nonnegativity, we know that for sufficiently small $\delta$
\[ \omega_i^\delta = \omega_i, \quad \varrho_i^\delta=\tilde\varrho_i^\delta=\varrho_i \quad \text{and} \quad F_i^\delta=F_i, \quad i=1,\dots,N. \]
Moreover, using (\ref{energy2}) we get the estimates
\[ \|u\|_{L^2(0,T;W^{1,2})}^2 + \sum_{i=1}^N\left(\varepsilon\|\nabla\varrho_{i}\|_{L^2((0,T)\times\Omega)}^2+\delta\|\varrho_{i}\|_{L^\beta((0,T)\times\Omega)}^\beta + \left\|\frac{1}{\sqrt{\varrho_i}}F_i\right\|_{L^2((0,T)\times\Omega)}\right)
\leq C, \]
where $C$ does not depend on $\varepsilon$ and $\delta$. In consequence, we also get
\[ \|\partial_t\varrho_i\|_{L^2(0,T;H^{-1})} \leq C. \]
Therefore, if $(\vec\varrho_{\varepsilon,\delta}, u_{\varepsilon,\delta})$ is a solution to (\ref{approx2}), then up to a subsequence
\[ \varrho_{1,\varepsilon,\delta},\dots,\varrho_{N,\varepsilon,\delta}, u_{\varepsilon,\delta} \rightharpoonup \varrho_{1,\varepsilon},\dots,\varrho_{N,\varepsilon}, u_\varepsilon \quad \text{in} \quad L^2(0,T;W^{1,2}) \]
and from the Aubin--Lions lemma
\[ \varrho_{i,\varepsilon,\delta} \to \varrho_{i,\varepsilon} \quad \text{in} \quad L^2((0,T)\times\Omega). \]
Thus passing to the limit in the weak formulation (note that the terms $\int_0^T\int_\Omega \varrho_{\varepsilon,\delta}^{\beta-2}\varrho_{i,\varepsilon,\delta}\varphi \;\dd x\dd t$ for $\varphi\in C_0^\infty((0,T)\times\Omega)$ are bounded independently of $\delta$) we obtain the weak solution to the system (\ref{approx}).

\section{End of the proof of Theorem \ref{positive_th}}\label{limit_sect}

In this section we perform the limit passage $\varepsilon\to 0$ with the additional assumption that all densities are bounded away from zero. First, note that we can rewrite $F_i$ in a following way:

\[ F_i = \sum_{j=1}^N C_{i,j}\nabla\frac{1}{m_j}\varrho_j^{\gamma_j} = \sum_{j=1}^N\tilde C_{i,j}\frac{\gamma_j}{\gamma_j-1}\frac{1}{m_j}\nabla\varrho_j^{\gamma_j-1}, \]
where
\[ \tilde C_{i,j} = \left\{ \begin{aligned}
    \frac{\varrho_i}{\varrho}\sum_{k\neq i}\varrho_k, \quad i=j, \\
    -\frac{\varrho_i\varrho_j}{\varrho}, \quad i\neq j
\end{aligned}\right.\]
is symmetric. The following transformation is also performed in a more general setting in Chapter 7 in \cite{giovangigli} and in \cite{DDGG} for a system with different pressure law.

Using the above relation we can write $F_1,\dots,F_{N-1}$ as a combination of $\nabla q_1,\dots,\nabla q_{N-1}$ for 
\[ q_i = \frac{\gamma_i}{\gamma_i-1}\frac{1}{m_i}\varrho_i^{\gamma_i-1}-\frac{\gamma_{i+1}}{\gamma_{i+1}-1}\frac{1}{m_{i+1}}\varrho_{i+1}^{\gamma_{i+1}-1}, \]
namely
$\displaystyle F_i = \sum_{j=1}^{N-1} b_{i,j}\nabla q_j $
with
\[ b_{i,j} = \left\{\begin{aligned}
    -\frac{\varrho_i}{\varrho}\sum_{k=1}^j\varrho_k, \quad j<i, \\
    \frac{\varrho_i}{\varrho}\sum_{k=j+1}^N\varrho_k, \quad j\geq i
\end{aligned}\right. \]
By performing elementary operations on the matrix $B=(b_{i,j})_{i,j=1,\dots,N-1}$, it is easy to see that $\det B = \frac{1}{\varrho}\varrho_1\dots\varrho_N$. Indeed, adding the verses $2,\dots,N-1$ to the first one and dividing it by $\frac{\varrho_N}{\varrho}$, we get the matrix
\[\begin{bmatrix}
    \varrho_1 & \varrho_1+\varrho_2 & \dots & \dots & \sum_{k=1}^{N-1}\varrho_k \\
    -\frac{\varrho_1\varrho_2}{\varrho} & \frac{\varrho_2}{\varrho}\sum_{k=3}^N\varrho_k & \dots & \dots & \frac{\varrho_2}{\varrho}\varrho_N \\
    -\frac{\varrho_1\varrho_3}{\varrho} & -\frac{\varrho_3}{\varrho}\sum_{k=1}^2\varrho_k & \dots & \dots & \frac{\varrho_3}{\varrho}\varrho_N \\
    \vdots & \vdots & \dots & \dots & \vdots \\
    -\frac{\varrho_1\varrho_{N-1}}{\varrho} & -\frac{\varrho_{N-1}}{\varrho}\sum_{k=1}^2\varrho_k & \dots & \dots & \frac{\varrho_{N-1}}{\varrho}\varrho_N
\end{bmatrix}\]
Now to the $i$-th verse we add the first one multiplied by $\frac{\varrho_i}{\varrho}$. After that operation, we obtain the triangular matrix with $\varrho_1,\varrho_2,\dots,\varrho_{N-1}$ on the diagonal. In conclusion
\[ \det B = \frac{\varrho_N}{\varrho}\cdot\varrho_1\dots\varrho_{N-1} 
\text{ \ \ as we claimed.} \]
In consequence,  matrix $B$ is invertible as $\varrho_1\dots\varrho_N\neq 0$ and if $\varrho_1\dots\varrho_N\geq c$ for some $c>0$, then 
\[ \sum_{i=1}^{N-1}|\nabla q_i|^2 \leq C\sum_{i=1}^{N-1}|F_i|^2 \]
for some constant $C$ depending on $c$ and $B^{-1}$. In particular, from the $L^2$ bound on $F_i$ we get
\[ \sum_{i=1}^{N-1}\|\nabla q_i\|_{L^2((0,T)\times\Omega)} \leq C. \]

Now, let us show the Lipschitz estimate on $\varrho_1,\dots,\varrho_N$ with respect to $q_1,\dots,q_{N-1},\varrho$. Let 
\[ G(z_1,\dots,z_N) = \left[\begin{gathered}
\frac{\gamma_1}{\gamma_1-1}\frac{1}{m_1}z_1^{\gamma_1-1} - \frac{\gamma_2}{\gamma_2-1}\frac{1}{m_2}z_2^{\gamma_2-1} \\
\frac{\gamma_2}{\gamma_2-1}\frac{1}{m_2}z_2^{\gamma_2-1} - \frac{\gamma_3}{\gamma_3-1}\frac{1}{m_3}z_3^{\gamma_3-1} \\
\vdots \\
\frac{\gamma_{N-1}}{\gamma_{N-1}-1}\frac{1}{m_{N-1}}z_{N-1}^{\gamma_{N-1}-1} - \frac{\gamma_N}{\gamma_N-1}\frac{1}{m_N}z_N^{\gamma_N-1} \\
z_1 + \dots + z_N
\end{gathered}\right]. \]
Then
\[ (\vec q, \varrho) = G(\vec\varrho\,). \]
We will now prove that $G$ is invertible. First, let us show two auxillary results:

\begin{prop}\label{det}
    $\det DG(\vec z)\neq 0$ for any $\vec z \in \mathcal{U} :=\{\vec z\in\mathbb{R}^N: z_1,\dots,z_N > 0\}$.
\end{prop}
\begin{proof}
    We have
    \[ DG(\vec z\,) = \begin{bmatrix}
        a_1 & -a_2 & 0 & \dots & 0 \\
        0 & a_2 & -a_3 & \dots & 0 \\
        \vdots & \dots & \ddots & \dots & \vdots \\
        0 & \dots & 0 & a_{N-1} & -a_N \\
        1 & \dots & \dots & \dots & 1
    \end{bmatrix}\]
    for $a_i = \frac{\gamma_i}{m_i}z_i^{\gamma_i-2}$. By elementary operations, we can transform this matrix into
    \[\begin{bmatrix}
        a_1 & -a_2 & 0 & \dots & 0 \\
        0 & a_2 & -a_3 & \dots & 0 \\
        \vdots & \dots & \ddots & \dots & \vdots \\
        0 & \dots & 0 & a_{N-1} & -a_N \\
        0 & \dots & \dots & 0 & 1+a_N\sum_{j=1}^{N-1}\frac{1}{a_j}
    \end{bmatrix}\]
    Therefore 
    \[ \det DG(\vec z) = \sum_{i=1}^N\prod_{j\neq i}a_j >0 
    \text{ \ \ \ \ \ \ for \ \ $z_1,\dots,z_N>0$} \]
\end{proof}

 In the next Proposition we analyze the codomain of $G$. For simplicity we assume that $\frac{\gamma_i}{\gamma_i-1}\frac{1}{m_i}=1$, otherwise we can additionally rescale the variables to get the same result. For clarity, below we present the precise formulation for only $3$ components, however we can proceed analogously for arbitrary $N$.

\begin{prop}[For $N=3$]
    $G(\mathcal{U})\subseteq \mathcal{V}:=\{(s_1,s_2,s_3): s_3>g(s_1,s_2)\}$, with 
    \[ g(s_1,s_2) = \left\{\begin{aligned}
        (-s_1)^{1/\alpha_2}+(-s_1-s_2)^{1/\alpha_3}, &\quad (s_1,s_2)\in I_1 \\
        s_1^{1/\alpha_1}+(-s_2)^{1/\alpha_3}, &\quad (s_1,s_2)\in I_2 \\
        (s_1+s_2)^{1/\alpha_1}+s_2^{1/\alpha_2}, &\quad (s_1,s_2)\in I_3
    \end{aligned}\right. \]
    for $\alpha_i=\gamma_i-1$ and
    \[ \begin{aligned}
        I_1 &= \{s_1<0, \; s_1+s_2<0 \}, \\
        I_2 &= \{s_1>0, \; s_2<0 \}, \\
        I_3 &= \{s_2>0, \; s_1+s_2>0. \}
    \end{aligned}\]
\end{prop}
\begin{proof}
    It is enough to examine the behavior of $G$ at $\partial\mathcal{U}$. We have
    \[ G(0,z_2,z_3) = \begin{bmatrix}
        -z_2^{\alpha_2}, \\
        z_2^{\alpha_2}-z_3^{\alpha_3}, \\
        z_2+z_3
    \end{bmatrix}.\]
    Therefore using the parametrisation $s_1=-z_2^{\alpha_2}$, $s_2=z_2^{\alpha_2}-z_3^{\alpha_3}$ we see that 
    \[ z_2+z_3 = |s_1|^{1/\alpha_2}+(|s_1|-s_2)^{1/\alpha_3}. \]
    Moreover, since $z_2,z_3\geq 0$, we have $(s_1,s_2)\in I_1$. Proceeding analogously for the remaining components of $\partial\mathcal{U}$, we get that $G(\partial\mathcal{U})$ is a graph of a function $g$.
\end{proof}
Note that in the case of arbitrary $N$ analogous calculations show that the function $g$ has a form 
\[ g(s_1,\dots,s_{N-1}) = \sum_{j=1}^{i-2}\left(\sum_{l=j}^{i-1}s_l\right)^{1/\alpha_j} + s_{i-1}^{1/\alpha_{i-1}}+(-s_i)^{1/\alpha_i} + \sum_{j=i+2}^N\left(-\sum_{l=i}^{j-1}s_l\right)^{1/\alpha_j} \]
for $(s_1,\dots,s_{N-1})\in I_i$, where
\[ I_i = \left\{\sum_{l=j}^{i-1}s_l \geq 0, \; j=1,\dots,i-1, \quad \text{and} \quad \sum_{l=i}^k s_l \leq 0, \; k=i,\dots,N-1\right\}. \]

The above propositions allows us to conclude invertibility of $G$. By Proposition \ref{det} $G$ is a local diffeomorphism. Then, since $\mathcal{V}$ is simply connected and $G\colon\mathcal{U}\to\mathcal{V}$ is proper, by the Hadamard's theorem (see e.g. Theorem 6.2.8 in \cite{implicit_book}), $G$ is invertible.

Since $G$ is a diffeomorphism, we know that
\[ \vec\varrho = G^{-1}(\vec q,\varrho) \]
is locally Lipschitz. Therefore in particular
\begin{equation}\label{p_lip} |p(\vec\varrho(t,x))-p(\vec\varrho(t,y))| \leq C|\varrho(t,x)-\varrho(t,y)| + C\sum_{j=1}^{N-1}|q_j(t,x)-q_j(t,y)|. \end{equation}

Having the above estimates, we will show the compactness of the total density $\varrho$. We use the following version of the compactness criterion from \cite{bresch-jabin,BJ2}:

\begin{prop}\label{BJ}
Let $\varrho_k$ be a sequence uniformly bounded in $L^p((0,T)\times\Omega)$ for $1\leq p<\infty$. Assume that $\{K_h\}_{h>0}$ is a family of positive, bounded functions on $\mathbb{R}^3$, satisfying:
\begin{itemize}
    \item $\displaystyle \forall_{\eta>0} \quad \sup_{h>0}\int_{\mathbb{R}^3} K_h(x)\mathbbm{1}_{\{|x|>\eta\}}\dd x <\infty$,
    \item $\|K_h\|_{L^1}\to \infty$ as $h\to 0$.
    \end{itemize}
    If $\partial_t\varrho_k$ is uniformly bounded in $L^q(0,T;W^{-1,p})$ for some $q>1$ and
    \[ \limsup_k\int_0^T\frac{1}{\|K_h\|_1}\iint_{\Omega\times\Omega} K_h(x-y)|\varrho_k(t,x)-\varrho_k(t,y)|^p\; \dd x\dd y\dd t \to 0 \quad \text{as} \quad h\to 0, \]
    then $\{\varrho_k\}_{k\in\mathbb{N}}$ is compact in $L^p((0,T)\times\Omega)$. Conversely, if $\{\varrho_k\}_{k\in\mathbb{N}}$ is compact in $L^p((0,T)\times\Omega)$, then the above quantity converges to $0$.
\end{prop}

Following \cite{bresch-jabin}, we choose the functions $K_h$ in the following way: \\
Let $K_h\colon \mathbb{R}^3\to\mathbb{R}$ be given by
\[ K_h(x) = \frac{1}{(|x|+h)^3} \quad \text{for} \quad |x|\leq\frac{1}{2} \]
 and let $K_h$ be independent of $h$ for $|x|>\frac{2}{3}$, zero outside the ball $B(0,\frac{3}{4})$ and such that $K_h\in C^\infty(\mathbb{R}^3)$. Denote also $\overline{K}_h=\frac{K_h}{\|K_h\|_1}$. An important property of $K_h$ is that $\|K_h\|_1 \sim |\log h|$.
To use the Proposition \ref{BJ}, we will estimate the quantity
\[ \mathcal{R}_h(t) = \frac{1}{\|K_h\|_1}\iint_{\Omega\times\Omega} K_h(x-y)|\varrho_{k}(t,x)-\varrho_{k}(t,y)|^2\dd x\dd y. \]
We need to estimate $\frac{\dd}{\dd t}\mathcal{R}_h(t)$. In the following calculations we drop the index $k$ where it does not raise confusion, and use the notation $f^x:= f(t,x)$. 

From the continuity equation, we get
\[\begin{aligned} \partial_t(\varrho^x-\varrho^y)^2 =& 2(\varrho^x-\varrho^y)\big(-\ddiv(\varrho^x u^x)+\ddiv(\varrho^y u^y) + \varepsilon\Delta\varrho^x-\varepsilon\Delta\varrho^y\big) \\
=& -\ddiv_x(u^x(\varrho^x-\varrho^y)^2) - \ddiv_y(u^y(\varrho_i^x-\varrho^y)^2) - (\ddiv u^x-\ddiv u^y)(\varrho^x+\varrho^y)(\varrho^x-\varrho^y) \\
&+ 2\varepsilon(\varrho^x-\varrho^y)(\Delta\varrho^x-\Delta\varrho^y).
\end{aligned}\]

Therefore

\[\begin{aligned}
\frac{\dd}{\dd t}\mathcal{R}_h(t) =& \frac{2}{\|K_h\|_1}\iint_{\Omega\times\Omega} \nabla K_h(x-y)(u^x-u^y)|\varrho^x-\varrho^y|^2 \dd x\dd y \\
& - 2\iint_{\Omega\times\Omega} \overline{K}_h(x-y)(\ddiv u^x-\ddiv u^y)(\varrho^x-\varrho^y)\varrho^x\dd x\dd y \\
& - 2\varepsilon\iint_{\Omega\times\Omega}\overline{K}_h(x-y)|\nabla\varrho^x-\nabla\varrho^y|^2 \dd x\dd y \\
=& A_1 + A_2 + A_3.
\end{aligned}\]

The last term has a good sign, so we just estimate it by $0$. The first term is estimated in the same way as in \cite{bresch-jabin}. By the definition of $K_h$, we have
\[ |\nabla K_h(z)|\leq \frac{CK_h(z)}{|z|}. \]
Using the inequality
\[ |f(x)-f(y)|\leq C|x-y|(M|\nabla f|(x) + M|\nabla f|(y)), \]
where $M$ denotes the maximal function, we arrive at
\[ A_1 \leq C\iint_{\Omega\times\Omega} \overline{K}_h(x-y)(M|\nabla u|^x + M|\nabla u|^y)|\varrho^x-\varrho^y|^2\dd x\dd y \]
and in consequence
\[ A_1 \leq C\iint_{\Omega\times\Omega} \overline{K}_h(x-y)M|\nabla u|^x|\varrho^x-\varrho^y|^2\dd x\dd y.\]

For the term $A_2$, we use the relation (\ref{d}) to write get
 \[ |\ddiv u^x-\ddiv u^y| \leq |d^x-d^y| + |p(\vec\varrho\,^x)-p(\vec\varrho\,^y)|. \]
Therefore by the Cauchy inequality
\[\begin{aligned}
A_2 &\leq C\iint_{\Omega\times\Omega}\overline{K}_h(x-y)\left(|d^x-d^y| + |p^x-p^y|\right)|\varrho^x-\varrho^y|\dd x\dd y \\
&\leq C\mathcal{R}_h(t) + C\iint_{\Omega\times\Omega}\overline{K}_h(x-y)|d^x-d^y|^2\;\dd x + C\iint_{\Omega\times\Omega}\overline{K}_h(x-y)|p^x-p^y|^2\;\dd x.
\end{aligned} \]

From (\ref{d_est}) we know that 
\[ \|\nabla d\|_{L^\infty(0,T;L^2)}\leq C. \]
Denoting by $Ed$ an $H^1(\mathbb{R}^3)$ extension of $d$, we have
\[\begin{aligned} \iint_{\Omega\times\Omega} \overline{K}_h(x-y)|d^x-d^y|^2\dd x\dd y &\leq \iint_{\mathbb{R}^3\times\mathbb{R}^3} \overline{K}_h(z)|Ed(t,y+z)-Ed(t,y)|^2\dd y\dd z \\
& \leq \int_0^1\int_{\mathbb{R}^3} \overline{K}_h(z)|z|^2\int_{\mathbb{R}^3} |\nabla Ed(t,y+sz)|^2\dd y\dd z\dd s \\
& \leq \|\nabla d\|_{L^\infty(0,T;L^2)}^2\frac{1}{|\log h|}\int_{\mathbb{R}^3} K_h(z)|z|^2\dd z.
\end{aligned} \]

In conclusion, since $\displaystyle\sup_{h}\int_{\mathbb{R}^3}K_h(z)|z|^2\dd z\leq C$, we get
\[ A_2 \leq C\mathcal{R}_h(t) + \frac{C}{|\log h|} + C\iint_{\Omega\times\Omega}\overline{K}_h(x-y)|p^x-p^y|^2\;\dd x\dd y. \]
Using (\ref{p_lip}) and Young inequality, we can further estimate it by
\[ A_2 \leq C\sum_{j=1}^{N-1}\iint_{\Omega\times\Omega}\overline{K}_h(x-y)|q_j^x-q_j^y|^2\dd x\dd y + C\mathcal{R}_h(t) +\frac{C}{|\log h|}. \]
Similarly as for $d$, we have
\[\begin{aligned} \iint_{\Omega\times\Omega} \overline{K}_h(x-y)|q_j^x-q_j^y|^2\dd x\dd y \leq \frac{\|\nabla q_j\|_{L^2}^2}{|\log h|}\int_{\mathbb{R}^3} K_h(z)|z|^2\dd z.
\end{aligned} \]
Therefore in the end we get
\[ A_2 \leq C\mathcal{R}_h(t) + \frac{C}{|\log h|}\left(\|\nabla d\|_{L^2(\Omega)}^2+\sum_{j=1}^{N-1}\|\nabla q_j\|_{L^2(\Omega)}^2\right). \]
In conclusion, 
\[\begin{aligned} \frac{\dd}{\dd t}\mathcal{R}_h(t) \leq & C\iint_{\Omega\times\Omega} \overline{K}_h(x-y)M|\nabla u^x|(\varrho^x-\varrho^y)^2\dd x\dd y+ C\mathcal{R}_h(t) \\
&+ \frac{C}{|\log h|}\left(\|\nabla d\|_{L^2(\Omega)}^2+\sum_{j=1}^{N-1}\|\nabla q_j\|_{L^2(\Omega)}^2\right) \end{aligned}\]
and combining all estimates, we arrive at
\[\begin{aligned} \frac{\dd}{\dd t}\mathcal{R}_h(t) \leq C\iint_{\Omega\times\Omega} \overline{K}_h(x-y)M|\nabla u^x|(\varrho^x-\varrho^y)^2\dd x\dd y + C\mathcal{R}_h(t) + \frac{C}{|\log h|}. \end{aligned}\]

For the first term, as the maximal function is bounded in $BMO$ (cf. Theorem 4.2 in \cite{bmo_max}), we use the following logarithmic inequality from \cite{mucha-rusin}:
\[ \left|\int_{\mathbb{R}^d} f(x)g(x)\dd x\right| \leq C\|f\|_{BMO}\|g\|_{L^1}\big(|\log\|g\|_{L^1}|+\log(1+\|g\|_{L^\infty})\big) \]
applied to $M|\nabla u^x|$ and $\displaystyle\int_{\Omega} \overline{K}_h(x-y)(\varrho^x-\varrho^y)^2\dd y$ extended by $0$ outside $\Omega$. We get
\begin{multline*} \iint_{\Omega\times\Omega} \overline{K}_h(x-y)M|\nabla u^x|(\varrho^x-\varrho^y)^2\dd x\dd y
\leq C\|\nabla u\|_{BMO} \mathcal{R}_h(t)\left(|\log \mathcal{R}_h(t)| + \log(1+C\|\varrho\|_{L^\infty}^2)\right). \end{multline*}

and after integrating over time, we end up with
\[\begin{aligned} \mathcal{R}_h(t) - \mathcal{R}_h(0) &\leq C\int_0^t \mathcal{R}_h(\tau)(1+|\log \mathcal{R}_h(\tau)|)\dd\tau + \frac{C}{|\log h|}\left(\|\nabla d\|_{L^2(0,T;L^2)}^2+\sum_{j=1}^{N-1}\|\nabla q_j\|_{L^2(0,T;L^2)}^2\right) \\
&\leq C\int_0^t \mathcal{R}_h(\tau)(1+|\log \mathcal{R}_h(\tau)|)\dd\tau + \frac{C}{|\log h|}.
\end{aligned}\]

Then
\[ \limsup_{k\to\infty}\sup_t R_h(t) \to 0 \quad \text{as} \quad h\to 0, \]
by the standard comparison criterion and the following Proposition (with $\varepsilon=|\log h|^{-1}$):
\begin{prop}
Let 
\[ 0\leq x_\varepsilon(t)\leq x_\varepsilon(0)+\varepsilon + \int_0^t x_\varepsilon(|\ln x_\varepsilon|+1)\dd\tau \]
with $x_\varepsilon(0)\to 0$ as $\varepsilon\to 0$.
Then $\displaystyle\sup_{t\in [0,T]} x_\varepsilon(t) \to 0$ as $\varepsilon\to 0$.
\end{prop}
\begin{proof}
Denote $\displaystyle y_\varepsilon(t) = x_\varepsilon(0)+\varepsilon + \int_0^t x_\varepsilon(|\ln x_\varepsilon|+1)\dd\tau$. As the function $x(|\ln x|+1)$ is increasing, we have
\[ \dot y_\varepsilon = x_\varepsilon(|\ln x_\varepsilon|+1) \leq y_\varepsilon(|\ln y_\varepsilon|+1). \]
Therefore from the standard comparison criterion for ODEs we have
\[ x_\varepsilon(t) \leq y_\varepsilon(t) \leq z_\varepsilon(t), \]
where $z_\varepsilon$ solves
\[ \dot z_\varepsilon = z_\varepsilon(|\ln z_\varepsilon|+1), \quad z_\varepsilon(0)=x_\varepsilon(0)+\varepsilon. \]
However, it is now easy to see that $\displaystyle \sup_{t\in [0,T]}z_\varepsilon(t)\to 0$ as $\varepsilon\to 0$. Indeed, as $\dot{z}_\varepsilon\geq 0$ we have $z_\varepsilon(t)\geq z_\varepsilon(0)$. Then, for the range of $t$'s small enough that $z(t)<1$, we get $|\ln z_\varepsilon(t)|\leq |\ln z_\varepsilon(0)|$. Then from Gronwall's lemma
\[ z_\varepsilon(t)\leq z_\varepsilon(0)e^{t+t|\ln z_\varepsilon(0)|} = e^t(z_\varepsilon(0))^{1-t}. \]
Therefore $\sup_{t\in [0,t_1]}z_\varepsilon(t)\to 0$ for some $t_1<1$. Note that in particular $z_\varepsilon(t_1)\to 0$, hence we are able to repeat these estimates on the consecutive intervals, and in consequence obtain the convergence on any finite interval $[0,T]$.
\end{proof}

In conclusion, the sequence $(\varrho_\varepsilon)$ is compact in $L^2((0,T)\times\Omega)$. Now, from the higher space regularity of $q_j$, we can extract compactness of all particular densities as well. For each $i$ we simply put 
\[\begin{aligned} \int_0^T\iint_{\Omega\times\Omega} \overline{K}_h(x-y) & |\varrho_{i,\varepsilon}(t,x)-\varrho_{i,\varepsilon}(t,y)|^2\dd x\dd y\dd t \\
\leq & C\int_0^T\iint_{\Omega\times\Omega}\overline{K}_h(x-y) |\varrho_\varepsilon(t,x)-\varrho_\varepsilon(t,y)|^2\dd x\dd y\dd t \\
&+ C\sum_{j=1}^{N-1}\int_0^T\iint_{\Omega\times\Omega}\overline{K}_h(x-y)|q_{j,\varepsilon}(t,x)-q_{j,\varepsilon}(t,y)|^2\dd x\dd y\dd t \\
\leq & C\int_0^T\iint_{\Omega\times\Omega} \overline{K}_h(x-y)|\varrho_\varepsilon(t,x)-\varrho_\varepsilon(t,y)|^2\dd x\dd y\dd t \\
&+ \frac{C}{|\log h|}\sum_{j=1}^{N-1}\|\nabla q_{j,\varepsilon}\|_{L^2((0,T)\times\Omega)}^2
\end{aligned}\]
and then the compactness of $(\varrho_{i,\varepsilon})$ follows immediately from Proposition \ref{BJ}.

\section{The general diffusing/non-diffusing case}

In the last section, we briefly present the necessary modifications for the system (\ref{diff_nondiff}). In that case, we need to divide the total density and the pressure into two parts, where the first part depends only on the diffusive components. In other words, $p=p^{(1)}+p^{(2)}$ and $\varrho=\varrho^{(1)}+\varrho^{(2)}$, where
\[ p^{(1)} = \sum_{i=1}^{N_1} p_i(\varrho_i), \quad \varrho^{(1)}=\sum_{i=1}^{N_1}\varrho_i \]
and $N_1$ is the number of diffusive components. With such decomposition of $p$ and $\varrho$, we define the fluxes only in terms of the first $N_1$ components, namely
\[ F_i = \nabla p_i - \frac{\varrho_i}{\varrho^{(1)}}\nabla p^{(1)}, \quad i=1,\ldots, N_1. \]
Note that with the above definition we preserve the important properties of $F_i$'s. In consequence, we are able to repeat the arguments from Sections 2 and 3 and derive the existence of solutions to the approximate system 
\begin{equation}\label{approx_gen}
    \begin{aligned}
    \partial_t\varrho_{i} + \ddiv(\varrho_{i} u) -\ddiv F_{i} &= \omega_{i}(\vec\varrho\,) + \varepsilon\Delta\varrho_i, \quad i=1,\dots,N_1 ,\\
    \partial_t\varrho_j + \ddiv(\varrho_ju) &= \omega_j(\vec\varrho\,) + \varepsilon\Delta\varrho_j, \quad j=N_1+1,\dots, N, \\
   -\mu\Delta u -\nabla((\mu+\lambda)\ddiv u) + \nabla p(\vec\varrho\,) &= 0
    \end{aligned}
\end{equation}
in the same way as before. To apply the reasoning from Section 4, we can perform the Bresch \& Jabin argument together for all the components. Defining
\[ \mathcal{R}_h(t) = \sum_{i=1}^N\mathcal{R}_h^i(t) = \sum_{i=1}^N\frac{1}{\|K_h\|_1}\frac{1}{m_i}\iint_{\Omega\times\Omega} K_h(x-y)|\varrho_{i,k}(t,x)-\varrho_{i,k}(t,y)|^2\dd x\dd y \]
and computing $\frac{\dd}{\dd t}\mathcal{R}_h$, it is clear that for the non-diffusing components the equivalents of the terms $A_1$, $A_2$, $A_4$ and $A_5$ can be dealt with in the same manner, whereas $A_3$ is just equal to $0$. In consequence, the computations from Section 4 give the strong convergence of densities for both diffusing and non-diffusing components. 

\subsection* {Acknowledgments:}
 The first (PBM)  author was partially supported by the National Science Centre grant no.  \\ 2022/45/B/ST1/03432 (Opus). The work of (MS) was supported by the National Science Centre grant no. 2022/45/N/ST1/03900 (Preludium). The work of \v S. N. was supported by Praemium Academiae of \v S. Ne\v casov\' a and by the Czech Science Foundation (GA\v CR) through project GA22-01591S. The Institute of Mathematics, Czech Academy of Sciences, is supported by RVO:67985840.

\bibliographystyle{abbrv}
\bibliography{biblio.bib}

\begin{thebibliography}{10}

\bibitem{bmo_max}
C.~Bennett, R.~A. DeVore, and R.~Sharpley.
\newblock {Weak-$L^\infty$ and $BMO$}.
\newblock {\em Annals of Mathematics}, 113(3):601--611, 1981.

\bibitem{B}
D.~Bothe.
\newblock {\em On the Maxwell-Stefan approach to multicomponent diffusion},
  volume~80 of {\em Parabolic problems, Progr. Nonlinear Differential Equations
  Appl.}
\newblock Birkhäuser/Springer Basel AG, Basel, 2011.

\bibitem{BD}
D.~Bothe and P.-E. Druet.
\newblock {On the structure of continuum thermodynamical diffusion fluxes-A
  novel closure scheme and its relation to the Maxwell-Stefan and the
  Fick-Onsager approach}.
\newblock {\em Internat. J. Engrg. Sci.}, 184(103818), 2023.

\bibitem{bothe-kroeger-warnecke}
D.~Bothe, M.~Kroeger, and H.-J. Warnecke.
\newblock A vof-based approach for the simulation of reactive mass transfer
  from rising bubbles.
\newblock {\em Fluid Dynamics \& Materials Processing}, 7:303--316, 01 2011.

\bibitem{BP}
D.~Bothe and M.~Pierre.
\newblock The instantaneous limit for reaction-diffusion systems with a fast
  irreversible reaction.
\newblock {\em Discrete Contin. Dyn. Syst. Ser. S}, 5(1):49--59, 2012.

\bibitem{BGS}
L.~Boudin, B.~Grec, and F.~Salvarani.
\newblock {A mathematical and numerical analysis of the Maxwell-Stefan
  diffusion equations}.
\newblock {\em Discrete Contin. Dyn. Syst. Ser. B}, 17(5):1427--1440.

\bibitem{bresch-jabin}
D.~Bresch and P.~Jabin.
\newblock {Global weak solutions of PDEs for compressible media: A compactness
  criterion to cover new physical situations}.
\newblock {\em Springer INdAM Series}, 17:33--54, 2017.

\bibitem{BJ2}
D.~Bresch and P.-E. Jabin.
\newblock {Global existence of weak solutions for compressible Navier--Stokes
  equations: Thermodynamically unstable pressure and anisotropic viscous stress
  tensor}.
\newblock {\em Annals of Mathematics}, 188(2):577 -- 684, 2018.

\bibitem{bulicek-pokorny}
M.~Bul\'i\v{c}ek, A.~J\"ungel, M.~Pokorn\'y, and N.~Zamponi.
\newblock {Existence analysis of a stationary compressible fluid model for
  heat-conducting and chemically reacting mixtures}.
\newblock {\em Journal of Mathematical Physics}, 63(5), 05 2022.

\bibitem{DG-M}
S.~De~Groot and P.~Mazur.
\newblock {\em Non-Equilibrium Thermodynamics}.
\newblock Dover Publication, New York, 1984.

\bibitem{DDGG}
W.~Dreyer, P.-E. Druet, P.~Gajewski, and C.~Guhlke.
\newblock {Analysis of improved Nernst-Planck-Poisson models of compressible
  isothermal electrolytes}.
\newblock {\em Z. Angew. Math. Phys.}, 71(4):68 pp., 2020.

\bibitem{D}
P.-E. Druet.
\newblock Incompressible limit for a fluid mixture.
\newblock {\em Nonlinear Anal. Real World Appl.}, 72(103859), 2023.

\bibitem{DJ}
P.-E. Druet and A.~Jüngel.
\newblock Analysis of cross-diffusion systems for fluid mixtures driven by a
  pressure gradient.
\newblock {\em SIAM J. Math. Anal.}, 52(2):2179--2197, 2020.

\bibitem{FPT}
E.~Feireisl, H.~Petzeltov\'a, and K.~Trivisa.
\newblock Multicomponent reactive flows: global-in-time existence for large
  data.
\newblock {\em Commun. Pure Appl. Anal.}, 7:1017--1047, 2008.

\bibitem{Gi}
V.~Giovangigli.
\newblock {\em Multicomponent flow modeling}.
\newblock Modeling and Simulation in Science Engineering and Technology.
  Birkh{\"a}user Boston Inc., Boston, MA, 1999.

\bibitem{giovangigli}
V.~Giovangigli.
\newblock {\em Multicomponent {Flow} {Modeling}}.
\newblock Modeling and {Simulation} in {Science}, {Engineering} and
  {Technology}. Birkhäuser Boston, Boston, MA, 1999.

\bibitem{GPZ}
V.~Giovangigli, M.~Pokorn\'y, and E.~Zatorska.
\newblock On the steady flow of reactive gaseous mixture.
\newblock {\em Analysis}, 5:319--341, 2015.

\bibitem{HMPW}
M.~Herberg, M.~Meyries, J.~Pruss, and M.~Wilke.
\newblock {Reaction-diffusion systems of Maxwell-Stefan type with reversible
  mass-action kinetics}.
\newblock {\em Nonlinear Anal.}, 159:264--284, 2017.

\bibitem{bubbles1}
G.~Juncu.
\newblock Unsteady heat and/or mass transfer from a fluid sphere in creeping
  flow.
\newblock {\em International Journal of Heat and Mass Transfer},
  44(12):2239--2246, 2001.

\bibitem{J}
A.~Jüngel.
\newblock {\em Entropy methods for diffusive partial differential equations}.
\newblock SpringerBriefs in Mathematics. Springer, Cham, 2016.

\bibitem{JS}
A.~Jüngel and I.~V. Stelzer.
\newblock {Existence analysis of Maxwell-Stefan systems for multicomponent
  mixtures}.
\newblock {\em SIAM J. Math. Anal.}, 45(4):2421–2440, 2013.

\bibitem{bubbles2}
A.~Koynov, J.~G. Khinast, and G.~Tryggvason.
\newblock Mass transfer and chemical reactions in bubble swarms with dynamic
  interfaces.
\newblock {\em AIChE Journal}, 51(10):2786--2800, 2005.

\bibitem{implicit_book}
S.~Krantz and H.~Parks.
\newblock {\em The implicit function theorem. History, theory, and
  applications.}
\newblock Birkh\"auser New York, NY, 01 2013.

\bibitem{Mucha}
P.~B. Mucha.
\newblock On cylindrical symmetric flows through pipe-like domains.
\newblock {\em J. Differential Equations}, 201(2):304--323, 2004.

\bibitem{Mucha-Pokorny}
P.~B. Mucha and M.~Pokorn\'{y}.
\newblock The rot-div system in exterior domains.
\newblock {\em J. Math. Fluid Mech.}, 16(4):701--720, 2014.

\bibitem{MPZ_SIMA}
P.~B. Mucha, M.~Pokorn\'y, and E.~Zatorska.
\newblock Heat-conducting, compressible mixtures with multicomponent diffusion:
  Construction of a weak solution.
\newblock {\em SIAM J. Math. Anal.}, 47:3747--3797, 2015.

\bibitem{degenerate-parabolic}
P.~B. Mucha, M.~Pokorný, and E.~Zatorska.
\newblock {Chemically reacting mixtures in terms of degenerated parabolic
  setting}.
\newblock {\em Journal of Mathematical Physics}, 54(7), 07 2013.

\bibitem{Mucha-Rautmann}
P.~B. Mucha and R.~Rautmann.
\newblock Convergence of {R}othe's scheme for the {N}avier-{S}tokes equations
  with slip conditions in 2{D} domains.
\newblock {\em ZAMM Z. Angew. Math. Mech.}, 86(9):691--701, 2006.

\bibitem{mucha-rusin}
P.~B. Mucha and W.~Rusin.
\newblock Zygmund spaces, inviscid limit and uniqueness of {Euler} flows.
\newblock {\em Communications in Mathematical Physics}, 280:831--841, 06 2008.

\bibitem{novotny-straskraba}
A.~Novotn\'y and I.~Stra\v{s}kraba.
\newblock {\em Introduction to the Mathematical Theory of Compressible Flow}.
\newblock Oxford Lecture Series in Mathematics and Its Applications. Oxford
  University Press, 2004.

\bibitem{Ot}
H.~\"Ottinger.
\newblock {\em Beyond Equilibrium Thermodynamics}.
\newblock Wiley Interscience, Hoboken, New Jersey, 2005.

\bibitem{PiPo1}
T.~Piasecki and M.~Pokorn\'y.
\newblock Weak and variational entropy solutions to the system describing
  steady flow of a compressible reactive mixture.
\newblock {\em Nonlinear Anal.}, 159:365--392, 2017.

\bibitem{PSZ1}
T.~Piasecki, Y.~Shibata, and E.~Zatorska.
\newblock On strong dynamics of compressible two-component mixture flow.
\newblock {\em SIAM J. Math. Anal.}, 51(4):2793 -- 2849, 2019.

\bibitem{PSZ2}
T.~Piasecki, Y.~Shibata, and E.~Zatorska.
\newblock {On the isothermal compressible multi-component mixture flow: The
  local existence and maximal $L^p$-$L^q$ regularity of solutions}.
\newblock {\em Nonlinear Anal.}, 189(111571), 2019.

\bibitem{Pr}
I.~Prigogine.
\newblock {\em Thermodynamics of Irreversible Processes}.
\newblock Interscience Publishers, John Wiley \& Sons, New York, 1967.

\bibitem{RaTa}
K.~Rajagopal and L.~Tao.
\newblock {\em Mechanics of Mixtures}.
\newblock World Scientific, Singapore, 1995.

\bibitem{Ro}
T.~Roub\'\i\v{c}ek.
\newblock {Incompressible ionized non-Newtonian fluid mixtures}.
\newblock {\em SIAM J. Math. Anal.}, 39:863--890, 2007.

\bibitem{sard}
A.~Sard.
\newblock The measure of the critical values of differentiable maps.
\newblock {\em Bull. Amer. Math. Soc.}, 48:883--890, 1942.

\bibitem{wu-yin-wang}
Z.~Wu, J.~Yin, and C.~Wang.
\newblock {\em Elliptic and parabolic equations}.
\newblock Hackensack, NJ: World Scientific, 2006.

\bibitem{Zatorska_2011}
E.~Zatorska.
\newblock On the steady flow of a multicomponent, compressible, chemically
  reacting gas.
\newblock {\em Nonlinearity}, 24(11):3267, oct 2011.

\bibitem{Za_1}
E.~Zatorska.
\newblock On the flow of chemically reacting gaseous mixture.
\newblock {\em J. Differential Equations}, 12(253):3471--3500, 2012.

\bibitem{zatorska2013}
E.~Zatorska.
\newblock Mixtures: Sequential stability of variational entropy solutions.
\newblock {\em Journal of Mathematical Fluid Mechanics}, 17:437--461, 2013.

\end{thebibliography}

\end{document}